\documentclass[12pt]{article}%
\usepackage{amsfonts, amsmath, amssymb}
\usepackage{subfig}
\usepackage{pstricks, pst-plot, pst-node, multido}
\usepackage{amsmath}
\usepackage{amsfonts}
\usepackage{amssymb}
\usepackage{graphicx}%
\providecommand{\U}[1]{\protect\rule{.1in}{.1in}}
\newtheorem{theorem}{Theorem}
\newtheorem{acknowledgement}[theorem]{Acknowledgement}

\newtheorem{corollary}[theorem]{Corollary}

\newtheorem{definition}[theorem]{Definition}

\newtheorem{lemma}[theorem]{Lemma}

\newtheorem{proposition}[theorem]{Proposition}
\newtheorem{remark}[theorem]{Remark}

\newenvironment{proof}[1][Proof]{\noindent\textbf{#1.} }{\ \rule{0.5em}{0.5em}}
\begin{document}

\title{Global bifurcation of planar and spatial periodic solutions from the polygonal
relative equilibria for the $n$-body problem\thanks{This is a corrected
version of the printed article. Mistakes on remarks have been corrected and
they have been moved to Section 8.4.}}
\author{Carlos Garc\'{\i}a-Azpeitia\thanks{Depto. Matem\'{a}ticas y Mec\'{a}nica,
IIMAS-UNAM, FENOMEC, Apdo. Postal 20-726, 01000 M\'{e}xico D.F. }, Jorge
Ize\footnotemark[2]}
\maketitle

\begin{abstract}
Given $n$ point masses turning in a plane at a constant speed, this paper
deals with the global bifurcation of periodic solutions for the masses, in
that plane and in space. As a special case, one has a complete study of $n$
identical masses on a regular polygon and a central mass. The symmetries of
the problem are used in order to find the irreducible representations, the
linearization, and with the help of the orthogonal degree theory, all the
symmetries of the bifurcating branches.

\end{abstract}

\section{Introduction}

Consider $n$ point masses turning at a constant speed in a plane around some
fixed point. A relative equilibrium of this configuration is a stationary
solution of the equations in the rotating coordinates. In this paper we
analyze the bifurcation of periodic solutions from a general relative
equilibrium. As a special case, we give a complete study of the polygonal
relative equilibrium where there are $n$ identical masses arranged on a
regular polygon and a central mass. This model was posed by Maxwell in order
to explain the stability of Saturn's rings. For this we give a full analysis
for the bifurcation of planar and spatial periodic solutions. We will prove
that, according to the value of the central mass, there are up to $2n$
branches of planar periodic solutions, with different symmetries, and up to
$n$ additional branches which are, if some non resonance condition is met, in
the space $\mathbb{R}^{3}$ with non trivial vertical components.

More precisely, in rotating coordinates and after a scaling in time, we look
for $2\pi$-periodic solutions $x_{j}(t)$ for $j=0,1,...,n$, where $j=0$
corresponds to the central body of mass $\mu$ while each element of the ring
has mass $1$, and frequency $\nu$. We shall prove (see the exact theorems and
graphs in the core of the paper) that, for each integer $k=1,...,n$, there are
one or two intervals (one unbounded) of values of the central mass with global
branches of periodic solutions bifurcating from the relative equilibrium. In fact:

For $k=2,...,n-2$ and for $k=1,n-1$ and $n>6$, there is an unbounded interval
of values of $\mu$ with two branches of planar periodic solutions (short and
long period) with the symmetries described below. For $k=1,n-1$ and
$n=3,4,5,6$ there is only one bifurcating branch in the unbounded interval,
while for $n=2$ there is one branch for any value of $\mu$.

For $k=2,...,n-2$, there is a bounded interval where one has one bifurcating
global branch for $\mu<\mu_{k}$, where $\mu_{k}$ is the value of the central
mass where one has the global branch of relative equilibria obtained in
\cite{GaIz11}. For $k=2,...,n/2$, there are two bifurcating branches for
$\mu_{k}<\mu<m_{0}$.

For $k=n$, and any $n>1$, there is a global branch of planar solutions
starting at $\nu=1$.

Finally, for $k=1,...,n$ and any $\mu$, there is a global branch of solutions
with a non trivial vertical coordinate, which is odd in time while the
horizontal coordinates are $\pi$-periodic, with the possible exception of a
finite number of resonances at a finite number of values of $\mu$'s.

By global branch, we mean that there is a continuum of solutions starting at
the ring configuration and a specific value of the frequency, which goes to
infinity in the norm of the solution or in the period $1/\nu$, or goes to
collision, or otherwise goes to other relative equilibria in such a way that
the sum of the jumps in the orthogonal degrees is zero. Except for a possible
finite number of $\mu$'s and $\nu$'s, for $\mu$ positive, bounded and
different from $\mu_{k}$, due to resonances, these solutions are geometrically different.

With respect to the symmetries, the solutions in the rotating system, with the
horizontal coordinates $u_{j}(t)$ and vertical coordinate $z_{j}(t)$, satisfy
the following properties, with $\zeta=2\pi/n$: the central body for $j=0$,
\begin{equation}
u_{0}(t)=e^{i\zeta}u_{0}(t+k\zeta),\qquad z_{0}(t)=z_{0}(t+k\zeta)\text{,}%
\end{equation}
and the $n$ equal bodies for $j=1,...,n$,%

\begin{equation}
u_{j}(t)=e^{ij\zeta}u_{n}(t+jk\zeta),\qquad z_{j}(t)=z_{n}(t+jk\zeta)\text{.}%
\end{equation}
Furthermore, the spatial solutions satisfy for $j=0,...,n$ that%
\begin{equation}
u_{j}(t)=u_{j}(t+\pi),\qquad z_{j}(t)=-z_{j}(t+\pi)\text{.}%
\end{equation}
These symmetries are related with choreographies of the $n$ body problem, see
Remark \ref{Cor}.

For $k=n$ and planar solutions, the bifurcation branch consists of solutions
with $u_{0}(t)=0$ and $u_{j}(t)=e^{ij\zeta}u_{n}(t)$, that is all the elements
move as they were in the ring configuration. This branch was constructed in an
explicit way in \cite{MS93}, by reducing the problem to a $6$-dimensional
dynamical system and a normal form argument.

For $k=n$ and spatial solutions, the branch is made of solutions where the
ring moves as a whole and the central body makes the contrary movement in
order to stabilize the forces, that is $z_{j}(t)=z_{n}(t)$ and $\mu
z_{0}(t)=-nz_{n}(t)$. This solution is called an oscillating ring in
\cite{MS93}.

For $k=n/2$ and planar solutions, the bifurcation has the central body fixed
at the center, and two polygons of $n/2$ bodies pulsing each one as a whole.

For $k=n/2$ and spatial solutions, one gets the well known Hip-Hop orbits.
This kind of solutions appears first in the paper \cite{DTW83} without the
central body. Later on, in \cite{MS93} for a big central body in order to
explain the pulsation of the Saturn ring, where they are called kink
solutions. Finally, there is a proof in \cite{BCPS06} when there is no central body.

For our particular setting, the linearization of the system at a critical
point is a ${2n}\times{2n}$ matrix, which is non invertible, due to the
rotational symmetry. These facts imply that the study of the spectrum of the
linearization is not an easy task and that the classical bifurcation results
for periodic solutions may not be applied directly. However, we shall use the
change of variables proved in our previous paper, \cite{GaIz11}, in order to
give not only this spectrum but also the consequences for the symmetries of
the solutions.

The present paper is the last part of our application of the orthogonal degree
to similar dynamical problems. In \cite{GaIz11}, we had a complete study of
the stationary problem, while in \cite{GaIz12} we have considered the
analytically simpler case of vortices and quasi-paralel filaments. Since this
paper is a continuation of \cite{GaIz11}, we shall use the results in that
paper, but we shall recall all the important notions.

The $n$-body problem has been the object of many papers, with different
techniques and different purposes. Therefore, there is a vast literature on
bifurcation of relative equilibria and on their stability, fewer on periodic
solutions. We shall mention the papers which are closer to our research, that
is \cite{MeHa91}, \cite{Ro00}, \cite{MeSc88}, \cite{Sc03}, among others. The
case of a satellite attracted by an array of $n$ masses was treated in
\cite{GaIz10}, and previously, for the case of $n=2$, by A. Maciejewski and S.
Rybicki in \cite{MR04}, using the orthogonal degree for the action of $SO(2)$.
See also \cite{Ga10} for details.

For few bodies one may use a normal form approach, which gives good local
information but is difficult to apply to large systems as the one we have
here. The same thing applies to the Lyapunov center theorem or the
decomposition with a central variety. For the case of relative equilibria, one
may consult \cite{FP08} and \cite{CD98}, among others.

Besides the orthogonal degree, which is designed for equivariant problems with
large symmetries, one may try to use other topological equivariant tools such
as the equivariant Conley index or equivariant variational indices which
effectively give a Weinstein-Moser theorem, \cite{Bar93}, but from a finite
orbit to a $T^{n}$-orbit, which is not the case here.

Variational techniques have been quite successful in treating the problem of
existence of special solutions such as the choreographies and the hip-hop
solutions. In particular, \cite{FT04}, \cite{F06} and \cite{F07}, classify all
the possible groups which give solutions which are minimizers of the action
without collisions, in the plane for the first paper and in space for the
other two. Thus, the issue is different from ours, since one has the proof of
the existence of a solution in the large, with a specific symmetry, but no
multiplicity results or localization of the solutions.

For the choreographies, following the seminal paper \cite{CM00}, with no
central mass, one has studies with more that 3 bodies and minima of the action
in \cite{Ch01} and \cite{BT04}, for instance.

In the case of hip-hop solutions, the first papers such as \cite{DTW83} and
\cite{MS93}, as well as \cite{BCPS06}, did not use variational methods. On the
other hand these methods were successful in \cite{CV00} and \cite{TV07}.

One would like to relate all the different solutions obtained by all these
methods. In the case of the analytical local solutions the linearization is
also used in the application of any degree theory, although in the later case
one may treat degenerate problems or large kernels. The relation between
topological solutions and variational solutions is not easy to establish. For
instance, it is one of the goal of \cite{CF08} to prove that the bifurcation
branch of hip-hop solutions may connect to variational solutions in the large.
In particular, they give a study of the vertical bifurcation with no central
mass and some conditions of non-resonance. The fact that there is no proof of
the uniqueness of the minimizers implies that there is no rigorous proof of a
global continuation and a connection of the hip-hop solutions to the eight choreography.

One of the main advantages of the orthogonal degree is that it applies to
problems which are not necessarily variational, but present conserved
quantities. Furthermore, deformations are easy to construct. From the
theoretical point of view the theory has to be extended to the action of
non-abelian groups and to abstract infinite dimensional spaces (in the case of
Hamiltonian systems of first order the degrees don't stabilize completely: see
\cite{IzVi03}, Lemma 3.6, p.264, and one has to use a simple reduction to a
finite number of modes, but without loss of information: see the section on
the Lyapunov-Schmidt reduction). See also\cite{Ry05} for the case of
gradients. Furthermore, the orthogonal degree gives global branches of
solutions, it takes into account all symmetries and proves the existence of
many solutions, even in the case of possible resonances, where the issue is to
know if the obtained solution is not a subharmonic (see the last section or
\cite{MST06}). As pointed out above, these solutions are different.

Section 2 is devoted to the mathematical setting of the problem, with the
symmetries involved in the general case and in the particular situation of the
regular polygon. We present in Sections 3-5, the preliminary results needed in
order to apply the orthogonal degree theory developed in \cite{IzVi03}, that
is the global Lyapunov-Schmidt reduction, the study of the irreducible
representations, for the general and the polygonal situations, with the change
of variables of \cite{GaIz11}, and the symmetries associated to these
representations. In Section 6, we prove our bifurcation results and, in
Section 7, we give the analysis of the spectrum in the general and the
polygonal situations, with the complete results on the type of solutions which
bifurcate from the relative equilibrium, in the plane and also in space. In
Section 8, we give some comments on the resonance condition, on the stability
and on the case of charged particles.

\section{Setting the problem}

Let $q_{j}(t)\in\mathbb{R}^{3}$ be the position of the body $j\in\{1,...,n\}$
with mass $m_{j}$. Define the $3\times3$ matrices $\bar{I}=diag(1,1,0)$ and
$\bar{J}=diag(J,0)$, where $J$ is the standard symplectic matrix in
$\mathbb{R}^{2}$. In rotating coordinates $q_{j}(t)=e^{\sqrt{\omega}t\bar{J}%
}u_{j}(t)$, Newton equations of the $n$ bodies are%
\[
m_{j}\ddot{u}_{j}+2m_{j}\sqrt{\omega}\bar{J}\dot{u}_{j}=\omega m_{j}\bar
{I}u_{j}-\sum_{i=1(i\neq j)}^{n}m_{i}m_{j}\frac{u_{j}-u_{i}}{\left\Vert
u_{j}-u_{i}\right\Vert ^{\alpha+1}}\text{.}%
\]

Define the vector $u$ as $(u_{1},...,u_{n})$, the matrix of masses
$\mathcal{M}$ as $diag(m_{1}I,...,m_{n}I)$ and the matrix $\mathcal{\bar{J}}$
as $diag(\bar{J},...,\bar{J})$. Then, Newton equations of the $n$ bodies in
vector form are%
\begin{align}
\mathcal{M}\ddot{u}+2\sqrt{\omega}\mathcal{M\bar{J}}\dot{u}  &  =\nabla
V(u)\text{ with }\label{D1.0.1}\\
V(u)  &  =\frac{\omega}{2}\sum_{j=1}^{n}m_{j}\left\Vert \bar{I}u_{j}%
\right\Vert ^{2}+\sum_{i<j}m_{i}m_{j}\phi_{\alpha}(\left\Vert u_{j}%
-u_{i}\right\Vert )\text{.}\nonumber
\end{align}
The function $\phi_{\alpha}(x)$ is defined such that $\phi_{\alpha}^{\prime
}(x)=-1/x^{\alpha}$. The gravitational force is the particular case $\phi
_{2}(x)=1/x$.

Critical points of the potential $V$ correspond to relative equilibria of the
$n$-body problem. This was part of the study done in \cite{GaIz11},
Proposition 1.

\begin{remark}
If one replaces, in the change of coordinates, the term $e^{\sqrt{\omega}%
t\bar{J}}$ with a complex factor $\varphi(t)$ (taking $q_{j}$ and $u_{j}$ as
complex functions instead of a planar vector), where $\varphi$ satisfies the
equation%
\[
{\varphi}^{\prime\prime}=-\omega\varphi/{\arrowvert\varphi\arrowvert}%
^{3}\text{,}%
\]
the equations become%

\begin{align}
(\mathcal{M}\ddot{u}+2(\varphi)^{\prime}/(\varphi)\mathcal{M}\dot
{u}){\arrowvert\varphi\arrowvert}^{3}  &  =\nabla V(u)\text{.}\nonumber
\end{align}
In particular, the stationary solutions of this system are the same solutions
studied in \cite{GaIz11}, and the solutions for $\varphi$ are those of the
Kepler problem, where $\omega$ is the central mass. Thus, the solutions for
$q_{j}$ are now ellipses, parabolas or hyperbolas, instead of circular orbits.
One may have also total collapse or solutions like $ct^{2/3} $, with
$\arrowvert c\arrowvert
=(9\omega/2)^{1/3}$. The bifurcation analysis for periodic solutions starting
near the relative equilibria may be performed in this slightly more general setting.
\end{remark}

To prove bifurcation of periodic solutions near a relative equilibrium, we
need to change variables as $x(t)=u(t/\nu)$. In this way, the $2\pi/\nu
$-periodic solutions of the differential equation are $2\pi$-periodic
solutions of the bifurcation operator%
\begin{align}
f  &  :H_{2\pi}^{2}(\mathbb{R}^{3n}\backslash\Psi)\times\mathbb{R}%
^{+}\rightarrow L_{2\pi}^{2}\label{D1.0.2}\\
f(x,\nu)  &  =-\nu^{2}\mathcal{M}\ddot{x}-2\sqrt{\omega}\nu\mathcal{\bar{J}%
M}\dot{x}+\nabla V(x)\text{.}\nonumber
\end{align}
The set $\Psi=\{x\in\mathbb{R}^{3n}:x_{i}=x_{j}\}$ is the collision set, when
two or more of the bodies collide, and $H_{2\pi}^{2}(\mathbb{R}^{3n}%
\backslash\Psi)$ is the open subset, consisting of the collision-free periodic
(and continuous) functions, of the Sobolev space $H^{2}(\mathbb{R}^{3n})$,%
\[
H_{2\pi}^{2}(\mathbb{R}^{3n}\backslash\Psi)=\{x\in H_{2\pi}^{2}(\mathbb{R}%
^{3n}):x_{i}(t)\neq x_{j}(t)\}.
\]

\begin{definition}
We define the action of $\Gamma= \mathbb{Z}_{2}\times SO(2)$ in $\mathbb{R}%
^{3n}$ as
\[
\rho(\kappa)=\mathcal{R}\text{ and }\rho(\theta)=e^{-\mathcal{\bar{J}\theta}%
},
\]
where $R=diag(1,1,-1)$ and $\mathcal{R}=diag(R,...,R),$ where $\kappa$ and
$\theta$ are elements of $\Gamma$.
\end{definition}

The group $\mathbb{Z}_{2}$ acts by reflection on the $z$-axis and the group
$SO(2)$ acts by rotation in the $(x,y)$-plane. Clearly, the potential $V$ is
invariant by the action of the group $\Gamma$. Consequently, the gradient
$\nabla V$ is $\Gamma$-orthogonal. This means that $\nabla V$ is $\Gamma
$-equivariant, $\rho(\gamma)\nabla V(x)=\nabla V(\rho(\gamma)x)$, and that
$\nabla V(x)$ is orthogonal to the infinitesimal generator,%
\[
A_{1}x=\frac{\partial}{\partial\theta}|_{\theta=0}e^{-\mathcal{\bar{J}\theta}%
}x=-\mathcal{\bar{J}}x\text{.}%
\]
These facts can be proved directly from the definitions.

Given that $\nabla V$ is $\Gamma$-equivariant and the equation is autonomous,
then the map $f$ is $\Gamma\times S^{1}$-equivariant, where the action of
$S^{1}$ is by time translation. Moreover, the infinitesimal generator of the
action of $S^{1}$ in time is $A_{0}x=\dot{x}$. Hence, the map $f$ is
$\Gamma\times S^{1}$-orthogonal because of the equalities%
\begin{align*}
\left\langle f(x),\dot{x}\right\rangle _{L_{2\pi}^{2}}  &  =-\frac{\nu^{2}}%
{2}\left\Vert \mathcal{M}^{1/2}\dot{x}\right\Vert ^{2}|_{0}^{2\pi}-2\nu
\sqrt{\omega}\left\langle \mathcal{\bar{J}M}^{1/2}\dot{x},\mathcal{M}%
^{1/2}\dot{x}\right\rangle _{L_{2\pi}^{2}}+V(x)|_{0}^{2\pi}=0\text{,}\\
\left\langle f(x),\mathcal{\bar{J}}x\right\rangle _{L_{2\pi}^{2}}  &  =\nu
^{2}\left\langle \mathcal{M}^{1/2}\dot{x},\mathcal{\bar{J}M}^{1/2}\dot
{x}\right\rangle _{L_{2\pi}^{2}}-\nu\sqrt{\omega}\left\Vert \mathcal{M}%
^{1/2}x\right\Vert ^{2}|_{0}^{2\pi}+\int_{0}^{2\pi}\left\langle \nabla
V,\mathcal{\bar{J}}x\right\rangle =0\text{.}%
\end{align*}

It is well known that the Newton equations, in fixed coordinates, are
invariant under the action of the group of symmetries $\mathbb{R}^{3}\rtimes
O(3)$ which gives the conservation laws of linear and angular momenta. In
rotating coordinates, the potential $V$ is invariant under the action of the
subgroup $\mathbb{Z}_{2}\times\mathbb{R}\times O(2)$.

Since the matrices $\mathcal{\bar{J}}$ and $\mathcal{R}$ anticommute,
$\mathcal{R\bar{J}}=-\mathcal{\bar{J}R}$, then the operator $f(x)$ is
equivariant under the action of the full group, described by
\begin{equation}
(T^{2}\cup\bar{\kappa}T^{2})\times(\mathbb{R}\cup\kappa\mathbb{R})\text{.}
\label{D1.0.3}%
\end{equation}

The actions of $(\theta,\varphi)\in T^{2}$ and $\bar{\kappa}$ are given by
$(\theta,\varphi)x=e^{-\mathcal{\bar{J}\theta}}x(t+\varphi)$ and $\bar{\kappa
}x(t)=\mathcal{R}x(-t)$, where $\mathcal{R}=diag(R,...,R)$ with
$R=diag(1,-1,1)$. The action of $\mathbb{R}\cup\kappa\mathbb{R}$ is given by
the spatial reflection $\kappa$, which was already defined, and the action of
$r\in\mathbb{R}$ is a translation on the $z$-axis, $rx=x+re$ with
$e_{3}=(0,0,1)$ and $e=(e_{3},...,e_{3})$.

The action of the group $\mathbb{R}$ gives a conservation law of the linear
momentum with respect to the spatial axis. Consequently, any spatial
translation of a relative equilibrium is also an equilibrium. Since the
orthogonal degree is defined only for compact abelian groups, \cite{IzVi03},
p. 70, then we have to restrict the full group of symmetries (\ref{D1.0.3}) to
the abelian subgroup $\Gamma\times S^{1}=T^{2}\times\left\langle
\kappa\right\rangle $. However, in order to follow with this procedure we have
to deal with the missing symmetries of $\mathbb{R}$ by hand. The easy way to
do it consists in restricting the map $f$ to the orthogonal space to
$e\in\mathbb{R}^{3n}$.

\begin{definition}
We define the map $f:\mathcal{W}\cap H_{2\pi}^{2}\rightarrow\mathcal{W}$ as
the restriction of the bifurcation map $f(x)$ to the space
\begin{equation}
\mathcal{W}=\{x\in L_{2\pi}^{2}(\mathbb{R}^{3n}\backslash\Psi):\int_{0}^{2\pi
}(x\cdot e)dt=0\}\text{.} \label{D1.0.4}%
\end{equation}

\end{definition}

From the previous remark, the potential $V$ is invariant under the action of
$\mathbb{R}$. In our approach, this means that $\nabla V(x)$ is orthogonal to
the infinitesimal generator of the action $\nabla V(x)\cdot e=0$, a fact which
may be proved directly from the definitions. Consequently, the map
$f:\mathcal{W}\cap H_{2\pi}^{2}\rightarrow\mathcal{W}$ is well defined since%
\[
\int_{0}^{2\pi}(f(x)\cdot e)dt=\int_{0}^{2\pi}(\nabla V(x)\cdot e)dt=0\text{.}%
\]

\subsection{General relative equilibria}

All relative equilibria are planar. Moreover, any spatial translation of a
relative equilibrium is also an equilibrium. Nevertheless, the only relative
equilibrium in $\mathcal{W}$ of the family of translations is the one in the
$(x,y)$-plane. Therefore, the positions $(u_{j},0)$ form a relative
equilibrium in $\mathcal{W}$ if they satisfy the relations%
\[
\omega u_{i}=\sum_{j=1~(j\neq i)}^{n}m_{j}\frac{u_{i}-u_{j}}{\left\Vert
u_{i}-u_{j}\right\Vert ^{\alpha+1}}\text{ with }u_{j}\in\mathbb{R}^{2}.
\]

\begin{remark}
Since the potential is homogenous, then any scaling of a relative equilibrium
is also an equilibrium. Hence, in principle we may choose the relative
equilibrium with $\omega=1$. However, we shall leave $\omega$ as a parameter
because it is easier to fix the norm of the polygonal equilibrium than the frequency.
\end{remark}

Since every relative equilibrium $x_{0}$ is planar, then the action of
$\kappa\in{\Gamma}$ leaves fixed $x_{0}$. Therefore, the isotropy subgroup of
$x_{0}$ in $\Gamma\times S^{1}$ is $\Gamma_{x_{0}}\times S^{1}$ with
$\Gamma_{x_{0}}= \left\langle \kappa\right\rangle =\mathbb{Z}_{2}$. Thus, the
orbit of $x_{0}$ is isomorphic to the group $S^{1}$: the orbit of the
equilibrium $x_{0}$ consists of all the rotations of $x_{0}$ in the $(x,y)$-plane.

Notice that the generator of the spatial group $\Gamma$ at $x_{0}$ is
$A_{1}x_{0}=-\mathcal{\bar{J}}x_{0}$, then $-\mathcal{\bar{J}}x_{0}$ is
tangent to the orbit and must be in the kernel of $D^{2}V(x_{0})$.

Now, since $\mathbb{Z}_{2}$ is in the isotropy subgroup of $x_{0}$, then
$D^{2}V(x_{0})$ is $\mathbb{Z}_{2}$-equivariant. From Schur's lemma, see
\cite{IzVi03}, p. 18, the matrix $D^{2}V(x_{0})$ must have a block diagonal
form. In the following proposition we prove directly this fact.

\begin{proposition}
\label{E1.1.1} Let $\mathcal{A}_{ij}$ be the $3\times3$ blocks such that
\[
D^{2}V(x_{0})=(\mathcal{A}_{ij})_{ij=1}^{n},
\]
then the matrices $\mathcal{A}_{ij}$ have the diagonal form
\[
\mathcal{A}_{ij}=diag(A_{ij},a_{ij}).
\]

Let $d_{ij}$ be the distance between $u_{i}$ and $u_{j}$, and let
$(x_{j},y_{j})$ be the components of $u_{j}$. For $i\neq j$, we have that
$a_{ij}=m_{i}m_{j}/d_{ij}^{\alpha+1}$ and
\begin{equation}
A_{ij}=-\frac{m_{i}m_{j}}{d_{ij}^{\alpha+3}}\left(
\begin{array}
[c]{cc}%
(\alpha+1)(x_{i}-x_{j})^{2}-d_{ij}^{2} & (\alpha+1)(x_{i}-x_{j})(y_{i}%
-y_{j})\\
(\alpha+1)(x_{i}-x_{j})(y_{i}-y_{j}) & (\alpha+1)(y_{i}-y_{j})^{2}-d_{ij}^{2}%
\end{array}
\right)  \text{.} \label{D1.1.1}%
\end{equation}
Moreover, for $i=j$ we have the equalities
\[
a_{ii}=-\sum_{j=1~(j\neq i)}^{n}a_{ij}\text{ and }A_{ii}=\omega m_{i}%
I-\sum_{j\neq i}A_{ij}.
\]

\end{proposition}

\begin{proof}
The function $\phi_{\alpha}(d_{ij})$ has the second derivative matrix%
\[
D_{u_{i}}^{2}\phi_{\alpha}(d_{ij})=\frac{\alpha+1}{d_{ij}^{\alpha+3}}\left(
\begin{array}
[c]{ccc}%
(x_{i}-x_{j})^{2} & (x_{i}-x_{j})(y_{i}-y_{j}) & 0\\
(x_{i}-x_{j})(y_{i}-y_{j}) & (y_{i}-y_{j})^{2} & 0\\
0 & 0 & 0
\end{array}
\right)  -\frac{I}{d_{ij}^{\alpha+1}}\text{.}%
\]
For $i\neq j$, since $\nabla_{u_{i}}\phi(d_{ij})=-\nabla_{u_{j}}\phi(d_{ij})$,
then%
\[
\mathcal{A}_{ij}=m_{i}m_{j}D_{u_{j}}\nabla_{u_{i}}\phi(d_{ij})=-m_{i}%
m_{j}D_{u_{i}}^{2}\phi(d_{ij})=diag(A_{ij},a_{ij})\text{.}%
\]
And for $i=j$ the matrix $\mathcal{A}_{ii}$ satisfies%
\[
\mathcal{A}_{ii}=\omega m_{i}\bar{I}+\sum_{j\neq i}m_{i}m_{j}D_{u_{i}}^{2}%
\phi(d_{ij})=\omega m_{i}\bar{I}-\sum_{j\neq i}\mathcal{A}_{ij}=diag(A_{ii}%
,a_{ii})\text{.}%
\]

\end{proof}

If we analyze the $n$-body problem in the plane instead of the space, then the
potential of the planar problem has the Hessian $D^{2}V(x_{0})=(A_{ij}%
)_{ij=1}^{n}$. Because of this fact, we say that $(A_{ij})_{ij=1}^{n}$ will
give the planar spectrum and $(a_{ij})_{ij=1}^{n}$ the spatial spectrum.

\subsection{The polygonal equilibrium}

Identifying the real and complex planes, a relative equilibrium is given by
the positions $(a_{j},0)\in\mathbb{R}^{3}$ with $a_{j}\in\mathbb{C}$. Defining
$\zeta=2\pi/n$, the polygonal equilibrium is formed by $n+1$ bodies: one body,
with mass $m_{0}=\mu$, at $a_{0}=0$, and one body for each $j\in\{1,...,n\}$,
with mass $m_{j}=1$, at $a_{j}=e^{ij\zeta}$.

It is easy to prove, for instance \cite{GaIz11}, Proposition 1, that the ring
configuration $\bar{a}=(a_{0},...,a_{n})$ is a relative equilibrium when
$\omega=\mu+s_{1}$, with
\[
s_{1}=\frac{1}{2^{\alpha}}\sum_{j=1}^{n-1}\frac{1}{\sin^{(\alpha-1)}%
(j\zeta/2)}\text{.}%
\]

In this paper, the norm of the polygonal equilibrium is fixed and $\omega
=\mu+s_{1}$ is a free parameter. In the paper \cite{GaIz11}, Theorem 24, we
have proved bifurcation of relative equilibria from the polygonal equilibrium
with the parameter $\mu$ . Now, we wish to analyze the bifurcation of planar
and spatial periodic solutions.

\begin{definition}
Let $S_{n}$ be the group of permutations of $\{1,...,n\}$. We define the
action of $S_{n}$ in $\mathbb{R}^{2(n+1)}$ as%
\[
\rho(\gamma)(x_{0},x_{1},...,x_{n})=(x_{0},x_{\gamma(1)},...,x_{\gamma(n)}).
\]

\end{definition}

As for the general case, the map $f$ is $(\mathbb{Z}_{2}\times SO(2))\times
S^{1}$-orthogonal. Moreover, since $n$ of the bodies have equal mass, then the
gradient $\nabla V$ is $S_{n}$-equivariant. Therefore, we may consider the map
$f$ as $\Gamma\times S^{1}$-orthogonal with respect to the abelian group%
\[
\Gamma=\mathbb{Z}_{2}\times\mathbb{Z}_{n}\times SO(2),
\]
where $\mathbb{Z}_{n}$ is the subgroup of $S_{n}$ generated by $\zeta(j)=j+1$.

Let $\mathbb{Z}_{n}(\zeta,\zeta)$ be the subgroup of $\Gamma$ generated by
$(\zeta,\zeta)\in\mathbb{Z}_{n}\times SO(2)$ with $\zeta=2\pi/n\in SO(2)$.
Since the actions of $(\zeta,\zeta)$, given by taking the $j$-th element in
the plane $x_{j}$ to $e^{-J\zeta}x_{j+1}$, for $j\in{1,.,.,.,n}$ and $x_{0}$
to $e^{-J\zeta}x_{0}$, and $\kappa\in\mathbb{Z}_{2}$, fixing the plane, leave
the equilibrium $\bar{a}$ fixed, then the isotropy group of $\bar{a}$ is
${\Gamma}_{\bar{a}}\times S^{1}$ with%
\[
\Gamma_{\bar{a}}=\mathbb{Z}_{2}(\kappa)\times\mathbb{Z}_{n}(\zeta,\zeta).
\]

\section{The Lyapunov-Schmidt reduction}

The orthogonal degree was defined only in finite dimension. There is a
theoretical difficulty for the extension to abstract infinite dimensional
spaces, in particular, for strongly indefinite problems, in the sense that the
degrees defined on finite dimensional approximations do not stabilize, as the
classical Leray-Schauder degree, when increasing the dimension, although the
changes are easy to compute: see, for instance, the case of Hamiltonian
systems treated in \cite{IzVi03}, Lemma 3.6, p. 264. In the present situation,
one has a very simple way of overcoming this difficulty, without any loss of
information and, at the same time, avoiding many technical problems of
functional analysis.

Thus, we must perform first a reduction to some finite space .

The bifurcation operator $f$, in Fourier series, is%
\[
f(x)=\sum_{l\in\mathbb{Z}}(l^{2}\nu^{2}\mathcal{M}x_{l}-2l\nu\sqrt{\omega
}(i\mathcal{\bar{J})M}x_{l}+g_{l})e^{ilt}\text{,}%
\]
where $x_{l}$ and $g_{l}$ are the Fourier modes of $x$ and $\nabla V(x)$,
which is relatively compact, by Sobolev embedding, with respect to the second
order terms. Since all the masses are non-zero, then the matrix $l^{2}\nu
^{2}\mathcal{M}I-2il\nu\sqrt{\omega}\mathcal{\bar{J}M}$ is invertible for all
large $l\nu$'s, for $\nu>0$. Therefore, if we choose a big enough $p$, we can
solve the modes $\left\vert l\right\vert >p$ in terms of the remaining $2p+1$
Fourier modes, on any bounded subset of $\mathcal{W}$ with frequencies $\nu$
uniformly bounded from below, with an application of the global implicit
function theorem, \cite{IzVi03}, p. 257.

In this way, we get that the zeros of the bifurcation operator $f$ correspond
to the zeros of the bifurcation function
\[
f(x_{1},x_{2}(x_{1},\nu),\nu)=\sum_{\left\vert l\right\vert \leq p}(l^{2}%
\nu^{2}\mathcal{M}x_{l}-2l\nu\sqrt{\omega}(i\mathcal{\bar{J})M}x_{l}%
+g_{l})e^{ilt}\text{,}%
\]
where $x_{1}$ corresponds to the $2p+1$ Fourier modes and $x_{2}$ to the
solution for the remaining modes. Then, the linearization of the bifurcation
function at some equilibrium $x_{0}$ is%
\[
f^{\prime}(x_{0})=\sum_{\left\vert l\right\vert \leq p}\left(  l^{2}\nu
^{2}\mathcal{M}-2l\nu\sqrt{\omega}(i\mathcal{\bar{J})M}+D^{2}V(x_{0})\right)
x_{l}e^{ilt}\text{.}%
\]

Since the bifurcation operator is real, the linearization is determined by the
blocks $M(l\nu)$ for $l\in\{0,...,p\}$, where $M(\nu)$ is the matrix%
\begin{equation}
M(\nu)=\nu^{2}\mathcal{M}-2\nu\sqrt{\omega}(i\mathcal{\bar{J})M}+D^{2}%
V(x_{0})\text{.} \label{D2.1.0}%
\end{equation}

Each one of the blocks $M(l\nu)$ corresponds to a Fourier mode:since the
orbits are real, one has that the $x_{l}$, for negative $l$'s are the
conjugates of $x_{-l}$. This is why we are considering $x_{l}\in\mathbb{C}$,
for $l\in\{0,...,p\}$. Since we restrict the operator $f(x)$ to the subspace
of $H_{2\pi}^{2}(\mathbb{R}^{3n}\backslash\Psi)$ orthogonal to $e$, then the
block for the first Fourier mode $M(0)$ must be restricted to the orthogonal
space to $e$. We shall denote this restriction as $M(0)^{\perp e}$.

\section{Irreducible representations}

The following step in order to use the orthogonal degree, \cite{IzVi03}
Theorem 3.1, p. 247, consists in identifying the irreducible representations,
that is, for the action of the isotropy group $\Gamma_{\bar{ a}}$ and
obtaining the decomposition of the matrix $M(\nu)$.

\subsection{The general relative equilibria}

We show first, for $n$ arbitrary bodies in the plane, how the matrix $M(\nu)$
decomposes into two blocks by the action of the isotropy group $\Gamma_{x_{0}%
}=\mathbb{Z}_{2}$. The two blocks correspond to the planar and spatial spectra
as in the previous decomposition of $D^{2}V(x_{0})$.

Since the group $\mathbb{Z}_{2}$ acts by reflection on the $z$-axis, then the
equivalent irreducible representations for $\mathbb{C}^{3n}$, leading to the
decomposition of $M(\nu)$, are%
\begin{align*}
V_{0}  &  =\{(x_{1},y_{1},0,...,x_{n},y_{n},0):x_{j},y_{j}\in\mathbb{C}%
\}\text{ and}\\
V_{1}  &  =\{(0,0,z_{1},...,0,0,z_{n}):z_{j}\in\mathbb{C}\}\text{.}%
\end{align*}
The action of $\kappa\in\mathbb{Z}_{2}$ on $V_{0}$ is given by $\rho
(\kappa)=I$ and on $V_{1}$ by $\rho(\kappa)=-I$. Define the isomorphisms
$T_{0}:\mathbb{C}^{2n}\rightarrow V_{0}$ and $T_{1}:\mathbb{C}^{n}\rightarrow
V_{1}$ as
\begin{align*}
T_{0}(x_{1},y_{1},...,x_{n},y_{n})  &  =(x_{1},y_{1},0,...,x_{n}%
,y_{n},0)\text{,}\\
T_{1}(z_{1},...,z_{n})  &  =(0,0,z_{1},...,0,0,z_{n})\text{,}%
\end{align*}
and define the linear orthogonal map $P$%
\[
P(x_{1},y_{1},z_{1},...,x_{n},y_{n},z_{n})=T_{0}(x_{1},y_{1},...,x_{n}%
,y_{n})+T_{1}(z_{1},...,z_{n})\text{.}%
\]
The transformation $P$ rearranges the planar and spatial coordinates.

Since $V_{0}$ and $V_{1}$ are subspaces of equivalent irreducible
representations, Schur's lemma implies that the matrix $M(\nu)$ must satisfy%
\[
P^{-1}M(\nu)P=diag(M_{0}(\nu),M_{1}(\nu))\text{,}%
\]
where $M_{0}$ and $M_{1}$ are respectively $2n\times2n$ and $n\times n$
matrices. In fact, we will exhibit explicitly the matrices $M_{0}$ and $M_{1}
$, in the following result:

\begin{proposition}
\label{E3.1.0} Define the matrices $\mathcal{M}_{1}=(m_{1},...,m_{n})$,
$\mathcal{M}_{2}=diag(m_{1},m_{1},...,m_{n},m_{n})$, and $\mathcal{J}%
=diag(J,...,J)$, then the blocks $M_{0}$ and $M_{1}$ are
\begin{align*}
M_{0}(\nu)  &  =\nu^{2}\mathcal{M}_{2}-2\nu\sqrt{\omega}(i\mathcal{J)M}%
_{2}+(A_{ij})_{ij=1}^{n},\\
M_{1}(\nu)  &  =\nu^{2}\mathcal{M}_{1}+(a_{ij})_{ij=1}^{n}.
\end{align*}

\end{proposition}

\begin{proof}
Since $P$ rearranges the planar and spatial coordinates and $D^{2}%
V(x_{0})=(\mathcal{A}_{ij})_{ij=1}^{n}$, with $\mathcal{A}_{ij}=diag(A_{ij}%
,a_{ij})$, then
\[
P^{-1}D^{2}V(x_{0})P=diag((A_{ij})_{ij=1}^{n},(a_{ij})_{ij=1}^{n})\text{.}%
\]
Moreover,%
\[
P^{-1}\mathcal{M}P=diag(\mathcal{M}_{2},\mathcal{M}_{1})\text{ and }%
P^{-1}\mathcal{\bar{J}}P=diag(\mathcal{J},0)\text{.}%
\]
Therefore,%
\begin{align*}
P^{-1}M(\nu)P  &  =\nu^{2}diag(\mathcal{M}_{2},\mathcal{M}_{1})-2\nu
\sqrt{\omega}diag((i\mathcal{J)M}_{1},0)\\
&  +diag((A_{ij})_{ij=1}^{n},(a_{ij})_{ij=1}^{n})\text{.}%
\end{align*}

\end{proof}

The block $M(0)$ must be restricted to the space orthogonal to $e$. Since
$e\in V_{1}$, with $T_{1}(1,...,1)=e$, then, in the new coordinates, the block
$M_{1}(0)$ must be restricted to the space orthogonal to $(1,...,1)$.

For the study of the planar $n$-body problem, this procedure gives the matrix
$M_{0}(\nu)$ instead of $M(\nu)$.

\subsubsection*{Planar representation}

The action of the isotropy group $\mathbb{Z}_{2}\times S^{1}$ on the space
$V_{0}$ is given by
\[
(\kappa,\varphi)x=e^{i\varphi}x\text{.}%
\]
Since $\kappa\in\mathbb{Z}_{2}$ fixes the points of $V_{0}$, then the space
$V_{0}$ has its isotropy subgroup generated by $\kappa$,%
\[
\mathbb{Z}_{2}(\kappa).
\]

Thus, solutions $x(t)$ with isotropy group $\mathbb{Z}_{2}$ must satisfy
\[
x(t)=\kappa x(t)=Rx(t).
\]
Consequently, orbits in $V_{0}$ are planar, $z_{j}(t)=0$.

\subsubsection{Spatial representation}

In $V_{1}$ the action of the group $\mathbb{Z}_{2}\times S^{1}$ is%
\[
(\kappa,\varphi)x=-e^{i\varphi}x.
\]
Since the action of $(\kappa,\pi)$ fixes the points of $V_{1}$, then the space
$V_{1}$ has its isotropy subgroup generated by $(\kappa,\pi)$,
\[
\mathbb{Z}_{2}(\kappa,\pi).
\]

Solutions $x(t)$ with isotropy group $\mathbb{Z}_{2}$ must satisfy%
\[
x(t)=(\kappa,\pi)x(t)=Rx(t+\pi).
\]
Consequently, orbits in $V_{1}$ satisfy the symmetry%
\begin{equation}
x_{j}(t)=x_{j}(t+\pi)\text{, }y_{j}(t)=y_{j}(t+\pi)\text{ and }z_{j}%
(t)=-z_{j}(t+\pi). \label{D3.2.0}%
\end{equation}

From the symmetry (\ref{D3.2.0}), we have that $z_{j}(t_{j})=0$ for some
$t_{j}$. This means that these solutions oscillate around the $(x,y)$-plane
and that the projection of the curve in the $(x,y)$-plane$\ $is $\pi
$-periodic. Furthermore, these solutions go around the projected $\pi
$-periodic curve once with the spatial coordinate $z_{j}(t)$ and once with
$-z_{j}(t)$. These solutions look like eights around the equilibrium points,
and we shall call them spatial eights.

\subsection{The polygonal equilibrium}

In the case of the polygonal equilibrium, the isotropy group has also the
action of $\mathbb{\tilde{Z}}_{n}$. Notice that the action of $(\zeta
,\zeta)\ $in $V_{k}$, $k=0,1$, is given by%
\begin{align*}
(\zeta,\zeta)T_{0}(x_{0},y_{0},...,x_{n},y_{n})  &  =T_{0}(e^{\mathcal{J}%
\zeta}(x_{0},y_{0},x_{\zeta(1)},...,y_{\zeta(n)}))\text{,}\\
(\zeta,\zeta)T_{1}(z_{0},...,z_{n})  &  =T_{1}(z_{0},z_{\zeta(1)}%
,...,z_{\zeta(n)})\text{.}%
\end{align*}
Thus, we expect a decomposition of the space $V_{k}$ into smaller irreducible
representations due to the action of $\mathbb{Z}_{n}(\zeta,\zeta)$.

\subsubsection{Planar representations}

We give first the decomposition of the space $V_{0}$ into smaller irreducible
representations. Since the representations for $n=2$ are different from those
for $n\geq3$, we shall restrict the study to the case $n\geq3$ and we shall
give some comments on the case $n=2$ at the end of the paper.

\begin{definition}
For $k\in\{2,...,n-2,n\}$, define the isomorphisms $T_{k}:\mathbb{C}%
^{2}\rightarrow W_{k}$ as%
\begin{align*}
T_{k}(w)  &  =(0,n^{-1/2}e^{(ikI+J)\zeta}w,...,n^{-1/2}e^{n(ikI+J)\zeta
}w)\text{ with}\\
W_{k}  &  =\{(0,e^{(ikI+J)\zeta}w,...,e^{n(ikI+J)\zeta}w):w\in\mathbb{C}%
^{2}\}\text{.}%
\end{align*}
For $k\in\{1,n-1\}$, define the isomorphism $T_{k}:$ $\mathbb{C}%
^{3}\rightarrow W_{k}$ as
\begin{align*}
T_{k}(\alpha,w)  &  =(v_{k}\alpha,n^{-1/2}e^{(ikI+J)\zeta}w,...,n^{-1/2}%
e^{n(ikI+J)\zeta}w)\text{ with}\\
W_{k}  &  =\{(v_{k}\alpha,e^{(ikI+J)\zeta}w,...,e^{n(ikI+J)\zeta}w):\alpha
\in\mathbb{C},w\in\mathbb{C}^{2}\}\text{,}%
\end{align*}
where $v_{1}$ and $v_{n-1}$ are the vectors%
\[
v_{1}=2^{-1/2}\left(  1,i\right)  \text{ and }v_{n-1}=2^{-1/2}\left(
1,-i\right)  .
\]

\end{definition}

We have proved, in the paper \cite{GaIz11}, p.3207, that the subspaces $W_{k}$
form the irreducible representations of $\mathbb{Z}_{n}(\zeta,\zeta)$ in
$\mathbb{C}^{2(n+1)}\simeq V_{0}$. Also, we proved, in \cite{GaIz11},
Proposition 5, that the action of $(\zeta,\zeta,\varphi)\in\mathbb{\tilde{Z}%
}_{n}\times S^{1}$ on $W_{kl}$, for the $l$-th Fourier mode, is given by
\[
\rho(\zeta,\zeta,\varphi)=e^{ik\zeta}e^{il\varphi}\text{.}%
\]
Consequently, the isotropy subgroup of $\Gamma_{\bar{a}}\times\mathbb{S}^{1}$
for the block $W_{k}$, for the fundamental mode, is generated by $\kappa
\in\mathbb{Z}_{2}$ and $\left(  \zeta,\zeta,-k\zeta\right)  \in\mathbb{Z}%
_{n}(\zeta,\zeta)\times S^{1}$, that is,%

\[
\mathbb{Z}_{n}\left(  \zeta,\zeta,-k\zeta\right)  \times\mathbb{Z}_{2}%
(\kappa).
\]

Since the subspaces $W_{k}$ are orthogonal to each other, then the linear map%
\[
Pw=\sum_{k=1}^{n}T_{k}(w_{k})
\]
is orthogonal, where $w=(w_{1},...,w_{n})$, with $w_{k}\in\mathbb{C}^{3}$ for
$k=1,n-1$ and $w_{k}\in\mathbb{C}^{2}$ for the remaining $k$'s.

Since the map $P$ rearranges the irreducible representations of $\mathbb{Z}%
_{n}(\zeta,\zeta)$, from Schur's lemma, we must have that%
\[
P^{-1}(A_{ij})_{ij=1}^{n}P=diag(B_{1},...,B_{n})\text{,}%
\]
where $B_{k}$ are matrices such that $(A_{ij})_{ij=1}^{n}T_{k}(w)=T_{k}%
(B_{k}w)$. In fact, in the paper \cite{GaIz11}, Propositions 18 and 19, we
gave the blocks $B_{k}$ for the general case $\alpha\geq1$. In the next
proposition, we state this result.

\begin{proposition}
\label{E3.2.0} Set $\alpha_{+}=(\alpha+1)/2$ and $\alpha_{-} = (\alpha-1)/2$.
Define
\[
\alpha_{k}=\alpha_{-}(s_{k+1}+s_{k-1})/2\text{, }\beta_{k}=\alpha_{+}%
(s_{k}-s_{1})\text{ and }\gamma_{k}=\alpha_{-}(s_{k+1}-s_{k-1})/2
\]
with%
\[
s_{k}=\frac{1}{2^{\alpha}}\sum_{j=1}^{n-1}\frac{\sin^{2}(kj\zeta/2)}%
{\sin^{\alpha+1}(j\zeta/2)}\text{.}%
\]
Then, the blocks satisfy $B_{n-k}=\bar{B}_{k}$ and they are given by%
\[
B_{k}=\alpha_{+}(I+R)\mu+(s_{1}+\alpha_{k})I-\beta_{k}R-\gamma_{k}iJ\text{,}%
\]
for $k\in\{2,...,n-2,n\}$, and%
\[
B_{1}=\left(
\begin{array}
[c]{ccc}%
\mu\left(  s_{1}+\mu+n\alpha_{-}\right)  & -\sqrt{n/2}\mu\alpha & -\sqrt
{n/2}\mu i\\
-\sqrt{n/2}\mu\alpha & s_{1}+\alpha_{1}+(\alpha+1)\mu & \alpha_{1}i\\
\sqrt{n/2}\mu i & -\alpha_{1}i & s_{1}+\alpha_{1}%
\end{array}
\right)  \text{.}%
\]

\end{proposition}

Since $M_{1}(\nu)$ is $\mathbb{\tilde{Z}}_{n}$-equivariant, from Schur's
lemma, we have that%
\[
P^{-1}M_{0}(\nu)P=diag(m_{01}(\nu),...,m_{0n}(\nu))\text{,}%
\]
where the matrices $m_{0k}(\nu)$ must satisfy $M_{0}T_{k}(w)=T_{k}(m_{0k}w)$.

\begin{proposition}
The matrices $m_{0k}(\nu)$ satisfy $m_{0k}(\nu)=\bar{m}_{0(n-k)}(-\nu)$ with%
\[
m_{0k}(\nu)=\nu^{2}I-2\nu\sqrt{\omega}(iJ)+B_{k}\text{ }%
\]
for $k\in\{2,...,n-2,n\}$, and%
\[
m_{01}(\nu)=\nu^{2}diag(\mu,I)-2\nu\sqrt{\omega}diag(\mu,iJ)+B_{1}.
\]

\end{proposition}

\begin{proof}
The matrix $\mathcal{M}_{2}$ is $diag(\mu,\mu,1,...,1)$. For $k\in
\{2,...,n-2,n\}$, the matrix $\mathcal{M}_{2}$\ satisfies $\mathcal{M}%
_{2}T_{k}(w)=T_{k}(w)$. Since $(i\mathcal{J})T_{k}(w)=T_{k}(iJw)$, then%
\[
M_{0}T_{k}(w)=\left(  \nu^{2}\mathcal{M}_{2}-2\nu\sqrt{\omega}\mathcal{M}%
_{2}(i\mathcal{J)}+(A_{ij})_{ij=1}^{n}\right)  T_{k}(w)=T_{k}(m_{0k}w)\text{.}%
\]

For $k=1$, the matrix $\mathcal{M}_{2}$ satisfies $\mathcal{M}_{2}%
T_{1}(w)=T_{1}(diag(\mu,1,1)w)$. Since $(iJ)v_{1}=v_{1}$, then $(i\mathcal{J}%
)T_{1}(w)=T_{1}(diag(1,iJ)w)$. Therefore,
\[
M_{0}T_{1}(w)=\left(  \nu^{2}\mathcal{M}_{2}-2\nu\sqrt{\omega}\mathcal{M}%
_{2}(i\mathcal{J)}+(A_{ij})_{ij=1}^{n}\right)  T_{1}(w)=T_{1}(m_{01}w).
\]
For $k=n-1$, using $(iJ)v_{2}=-v_{2}$ and a similar argument, we may prove
that%
\[
m_{0(n-1)}(\nu)=\nu^{2}diag(\mu,I)-2\nu\sqrt{\omega}diag(-\mu,iJ)+B_{n-1}%
\text{.}%
\]
The equalities $m_{0(n-k)}(\nu)=\bar{m}_{0k}(-\nu)$ follow from $B_{n-k}%
=\bar{B}_{k}$.
\end{proof}

\subsubsection{Spatial representation}

We analyze now the decomposition of $V_{1}$.

\begin{definition}
For $k\in\{1,...,n-1\}$, we define $T_{k}:\mathbb{C}\rightarrow W_{k}$ as%
\begin{align*}
T_{k}(w)  &  =(0,n^{-1/2}e^{ik\zeta}w,...,n^{-1/2}e^{nik\zeta}w)\text{ with
}\\
W_{k}  &  =\{(0,e^{ik\zeta}w,...,e^{nik\zeta}w):w\in\mathbb{C}\}\text{.}%
\end{align*}
And, for $k=n$, we define $T_{n}:\mathbb{C}^{2}\rightarrow W_{n}$ as%
\begin{align*}
T_{n}(\alpha,w)  &  =(\alpha,n^{-1/2}w,...,n^{-1/2}w)\text{ with}\\
W_{n}  &  =\{(\alpha,w,...,w):\alpha,w\in\mathbb{C}\}.
\end{align*}

\end{definition}

The action of $(\zeta,\zeta)$ in the subspace $W_{k}$ is given by
\[
(\zeta,\zeta)T_{k}(w)=T_{k}(e^{ik\zeta}w).
\]
Consequently, the decomposition in irreducible representations of the space
$\mathbb{C}^{n+1}\simeq V_{1}$, for the group $\mathbb{Z}_{n}(\zeta,\zeta)$,
are the subspaces $W_{k}$. The actions of the elements $\kappa$, $(\zeta
,\zeta)$ and $\varphi\in S^{1}$ in $W_{k}$ are%
\[
\rho(\kappa)=-1,\rho(\zeta,\zeta)=e^{ik\zeta}\text{ and }\rho(\varphi
)=e^{il\varphi}.
\]
Since the elements $(\zeta,\zeta,-k\zeta)\in\mathbb{Z}_{n}(\zeta,\zeta)\times
S^{1}$ and $(\kappa,\pi)\in\mathbb{Z}_{2}(\kappa)\times S^{1}$ act trivially
on $W_{k}$, then the isotropy group of $W_{k}$ is generated by $(\zeta
,\zeta,-k\zeta)\ $and $(\kappa,\pi)$,%
\begin{equation}
\mathbb{Z}_{n}\left(  \zeta,\zeta,-k\zeta\right)  \times\mathbb{Z}_{2}%
(\kappa,\pi).
\end{equation}

Since the spaces $W_{k}$ are orthogonal to each other, then the linear map
\[
Pw=\sum_{k=1}^{n}T_{k}(w_{k})
\]
is orthogonal, where $w=(w_{1},...,w_{n})$, with $w_{n}\in\mathbb{C}^{2}$ and
$w_{k}\in\mathbb{C}^{1}$ for the remaining $k$'s.

Since $M_{1}(\nu)$ is $\mathbb{\tilde{Z}}_{n}$-equivariant, by Schur's lemma,
the matrix $M_{1}(\nu)$ must satisfy%
\[
P^{-1}M_{1}(\nu)P=diag(m_{11}(\nu),...,m_{1n}(\nu))\text{,}%
\]
where the blocks $m_{1k}(\nu)$ are such that $M_{1}T_{k}(w)=T_{k}(m_{1k}w)$
for $k\in\{1,...,n\}$.

\begin{proposition}
\label{E3.2.1}We have $m_{1k}(\nu)=\nu^{2}-(\mu+s_{k})$ for $k\in
\{1,...,n-1\}$, and%
\[
m_{1n}(\nu)=\left(
\begin{array}
[c]{cc}%
\mu(\nu^{2}-n) & \sqrt{n}\mu\\
\sqrt{n}\mu & \nu^{2}-\mu
\end{array}
\right)  .
\]

\end{proposition}

\begin{proof}
We denote the coordinate $w_{i}\in\mathbb{R}^{2}$ of the vector $w=(w_{0}%
,...,w_{n})$ by $[w]_{i}$. For $k\in\{1,...,n-1\}$, if $l\neq0$, from the
definition%
\[
\lbrack(a_{ij})T_{k}(w)]_{l}=n^{-1/2}\sum_{j=1}^{n}a_{lj}e^{ijk\zeta}w\text{.}%
\]
Since $a_{lj}=m_{l}m_{j}/d_{lj}$, with $d_{lj}^{2}=4\sin^{2}((l-j)\zeta/2)$
for $l,j\in\{1,...,n\}$, then $a_{lj}=a_{n(j-l)}$, with $(j-l)\in\{1,...,n\}$,
modulus $n$. From the equality $a_{lj}e^{ijk\zeta}=e^{ilk\zeta}\left(
a_{n(j-l)}e^{i(j-l)k\zeta}\right)  $, we have that%
\[
\lbrack(a_{ij})T_{k}(w)]_{l}=n^{-1/2}e^{ilk\zeta}\left(  \sum_{j=1}^{n}%
a_{nj}e^{ijk\zeta}\right)  w=\left[  T_{k}(b_{k}w)\right]  _{l}\text{,}%
\]
where $b_{k}$ is the sum between parentheses.

In order to calculate $b_{k}$, notice that $a_{nn}=-\sum_{j=0}^{n-1}a_{nj}$
and $a_{n0}=\mu$, then
\[
b_{k}=-\mu+\sum_{j=1}^{n-1}\left(  e^{ijk\zeta}-1\right)  a_{nj}.
\]
Since $a_{nj}=1/d_{nj}^{\alpha+1}$, then%
\[
\sum_{j=1}^{n-1}\left(  e^{ijk\zeta}-1\right)  a_{nj}=-\sum_{j=1}^{n-1}%
\frac{2\sin^{2}(kj\zeta/2)}{2^{\alpha+1}\sin^{\alpha+1}(j\zeta/2)}%
=-s_{k}\text{.}%
\]
Thus, we have $(a_{ij})T_{k}(w)=T_{k}(b_{k}w)$, with $b_{k}=-(\mu+s_{k})$.

If $l\neq0$, then,
\[
\lbrack(a_{ij})T_{n}(\alpha,w)]_{l}=a_{l0}\alpha+n^{-1/2}\sum_{j=1}^{n}%
a_{lj}w\text{,}%
\]
with $\alpha,w\in\mathbb{C}$. Using the equalities $a_{l0}=\mu$ and
$\sum_{j=1}^{n}a_{lj}=-a_{l0}$, we have
\begin{equation}
\lbrack(a_{ij})T_{n}(\alpha,w)]_{l}=n^{-1/2}(\sqrt{n}\mu\alpha-\mu w).
\label{D3.2.1}%
\end{equation}

If $l=0$, then
\[
\lbrack(a_{ij})T_{n}(\alpha,w)]_{0}=a_{00}\alpha+n^{-1/2}\sum_{j=1}^{n}%
a_{0j}w.
\]
Using the equality $a_{00}=-n\mu$, we have%
\begin{equation}
\lbrack(a_{ij})T_{n}(\alpha,w)]_{0}=-n\mu\alpha+\sqrt{n}\mu w. \label{D3.2.2}%
\end{equation}
From the equalities (\ref{D3.2.1}) and (\ref{D3.2.2}), we conclude that%
\[
(a_{ij})T_{n}(\alpha,w)=T_{n}\left(  \left(
\begin{array}
[c]{cc}%
-n\mu & \sqrt{n}\mu\\
\sqrt{n}\mu & -\mu
\end{array}
\right)  \left(
\begin{array}
[c]{c}%
\alpha\\
w
\end{array}
\right)  \right)  .
\]

Here the matrix $\mathcal{M}_{1}$ is $diag(\mu,1,...,1)$. Then, the statements
of the proposition follow from the fact that $\mathcal{M}_{1}T_{k}%
(w)=T_{k}(w)$ for $k\in\{1,...,n-1\}$, and $\mathcal{M}_{1}T_{n}%
(\alpha,w)=T_{n}(\mu\alpha,w)$.
\end{proof}

Since the block $M_{1}(0)$ is restricted to the space orthogonal to
$(1,...,1)$, and since $T_{n}(1,\sqrt{n})=(1,...,1)$, then the block
$m_{1n}(0)$ must be restricted to the subspace orthogonal to $(1,\sqrt{n})$.

As we expected, the eigenvalues of $m_{1n}(0)$ are $0$ and $-(n+1)\mu$ with
eigenvectors $(1,\sqrt{n})$ and $(\sqrt{n},-1)$ respectively. Therefore, the
matrix $m_{1n}(0)$, on the subspace orthogonal to $(1,\sqrt{n})$, is
equivalent to%
\[
m_{1n}(0)^{\perp}=-(n+1)\mu.
\]
The map $m_{1n}(0)^{\perp}$ is invertible for $\mu\neq0$. For $\mu=0$ the
central body has zero mass, and, in this case, the coordinate corresponding to
$m_{1n}(0)^{\perp}$ may be taken away.

\section{Symmetries}

Before we show the existence of bifurcation points, we wish to describe the
symmetries of the solutions. That is, solutions with isotropy group
$\mathbb{Z}_{n}\left(  \zeta,\zeta,-k\zeta\right)  \times\mathbb{Z}_{2}%
(\kappa)$ for the blocks $m_{0k}(\nu)$ and $\mathbb{Z}_{n}\left(  \zeta
,\zeta,-k\zeta\right)  \times\mathbb{Z}_{2}(\kappa,\pi)$ for the blocks
$m_{1k}(\nu)$.

The coordinates $(x_{j},y_{j},z_{j})$, of the body $j\in\{0,...,n\}$, will
denote the components of the solution $x(t)$.

\subsection{Planar solutions}

In this part we describe the symmetries of the isotropy group $\mathbb{Z}%
_{n}\left(  \zeta,\zeta,-k\zeta\right)  \times\mathbb{Z}_{2}(\kappa).$. As we
have seen, due to the group $\mathbb{Z}_{2}$, these solutions must be planar.
We shall identify the real and complex planes by $u_{j}=x_{j}+iy_{j}$.

Now, since $(\zeta,\zeta,-k\zeta)$ generate $\mathbb{Z}_{n}$, the solutions
with isotropy group $\mathbb{Z}_{n}$ must satisfy
\[
u_{j}(t)=e^{-i\zeta}u_{\zeta(j)}(t-k\zeta).
\]

\begin{remark}
Notice that, if $u_{j}(t)$ is a solution, for $\nu$, with symmetry
$\mathbb{Z}_{n}\left(  \zeta,\zeta,-k\zeta\right)  $, then $u_{j}(-t)$ is a
solution, for $-\nu$, with symmetry $\mathbb{Z}_{n}\left(  \zeta,\zeta
,k\zeta\right)  $. In fact, the bifurcation phenomena are related by the fact
that $m_{0(n-k)}(\nu)=\bar{m}_{0k}(-\nu)$.
\end{remark}

In order to describe the symmetries of the group $\mathbb{Z}_{n}\left(
\zeta,\zeta,-k\zeta\right)  $ we need the following definition.

\begin{definition}
For each fixed $k$, let $h$ be the maximum common divisor of $n$ and $k$. We
define%
\[
\bar{n}=n/h\text{ and }\bar{k}=k/h.
\]

\end{definition}

For the central body $u_{0}$ one has the following symmetries.

\begin{proposition}
If $h>1$, the central body remains at the center $u_{0}(t)=0$. If $h=1$, the
central body satisfies
\[
u_{0}(t+\zeta)=e^{-ik^{\prime}\zeta}u_{0}(t)\text{,}%
\]
were $k^{\prime}$ is such that $k^{\prime}k=1$, modulus $n$.
\end{proposition}

\begin{proof}
Since $\zeta(0)=0$, the central body has the symmetry%
\[
u_{0}(t)=e^{-i(l\zeta)}u_{0}(t-k(l\zeta)).
\]
For $h>1$ take $l=\bar{n}$. Since $k(\bar{n}\zeta)=2\pi\bar{k}$ and $\bar
{n}\zeta=2\pi/h$, then the central body satisfies%
\[
u_{0}(t)=e^{-i(\bar{n}\zeta)}u_{0}(t-k(\bar{n}\zeta))=e^{-i(2\pi/h)}%
u_{0}(t)=0\text{.}%
\]
For $h=1$ take $l=k^{\prime}$, then the central body satisfies%
\[
u_{0}(t)=e^{-i(k^{\prime}\zeta)}u_{0}(t-k(k^{\prime}\zeta))=e^{-i(k^{\prime
}\zeta)}u_{0}(t-\zeta)\text{.}%
\]

\end{proof}

In order to describe the symmetries of the $n$ bodies with equal masses, we
use the notation $u_{j}=u_{j+kn}$, for $j\in\{1,...,n\}$. Then, $\zeta
(j)=j+1$, and the $n$ bodies satisfy%
\[
u_{j+1}(t)=e^{\ ij\zeta}u_{1}(t+jk\zeta).
\]
Thus, each one of the $n$ bodies follows the same planar curve, but with a
different phase and with some rotation in the $(x,y)$-plane.

Now, let us show some examples of these symmetries.

For $k=n$, the symmetries are%
\[
u_{0}(t)=0\text{ and }u_{j+1}(t)=e^{\ ij\zeta}u_{1}(t)\text{.}%
\]
Thus, the central body remains at the center, and the other $n$ bodies form a
$n$-polygon at any time: see the figure for $n=3$.

\begin{figure}[h]
\centering
\subfloat[Symmetries of $\mathbb{\tilde{Z}}_{n}(1)$.] {\
\begin{pspicture}(-2.5,-2.5)(2.5,2.5)\SpecialCoor
\psdots[dotstyle=o](0,0)(2;0)(2;120)(2;240)
\psellipticarc[linestyle=dashed](2;0)(.5,.3){120}{0}
\psellipticarc{*->}(2;0)(.5,.3){0}{120}
\rput{120}{
\psellipticarc[linestyle=dashed](2;0)(.5,.3){240}{120}
\psellipticarc{*->}(2;0)(.5,.3){120}{240}}
\rput{240}{
\psellipticarc[linestyle=dashed](2;0)(.5,.3){0}{240}
\psellipticarc{*->}(2;0)(.5,.3){240}{0}}
\psline[linestyle=dashed](.5;0)(.5;120)(.5;-120)
\psline{*->}(.5;0)(.5;-120)
\end{pspicture}
}
\subfloat[Symmetries of $\mathbb{\tilde{Z}}_{n}(3)$.] {
\begin{pspicture}(-2.5,-2.5)(2.5,2.5)\SpecialCoor
\psdots[dotstyle=o](2;0)(2;120)(2;240)
\psellipticarc[linestyle=dashed](2;0)(.5,.3){120}{0}
\psellipticarc{*->}(2;0)(.5,.3){0}{120}
\rput{120}{
\psellipticarc[linestyle=dashed](2;0)(.5,.3){120}{0}
\psellipticarc{*->}(2;0)(.5,.3){0}{120}}
\rput{240}{
\psellipticarc[linestyle=dashed](2;0)(.5,.3){120}{0}
\psellipticarc{*->}(2;0)(.5,.3){0}{120}}
\psdots(0,0)
\NormalCoor\end{pspicture}
}\caption{For $n=3$.}%
\end{figure}

For $k=1$, the symmetries are%
\[
u_{0}(t+\zeta)=e^{-i\zeta}u_{0}(t)\text{ and }u_{j+1}(t)=e^{\ ij\zeta}%
u_{1}(t+j\zeta).
\]
Therefore, the central body is determined by the time interval $[0,\zeta)$,
and the rest of its orbit is given by rotations. The other $n$ bodies follow
the orbit of one of them, but with a synchronization between the change of
phase and the $(x,y)$-rotation: see the figure for $n=3$. \begin{figure}[h]
\centering
\subfloat[Symmetries of $\mathbb{\tilde{Z}}_{n}(1)$.] {
\begin{pspicture}(-2.5,-2.5)(2.5,2.5)\SpecialCoor
\psdots[dotstyle=o](2;0)(2;72)(2;146)(2;219)(2;292)(0,0)
\psline[linestyle=dashed](.5;0)(.5;72)(.5;146)(.5;219)(.5;292)
\psline{*->}(.5;0)(.5;292)
\psellipticarc[linestyle=dashed](2;0)(.5,.3){0}{292}
\psellipticarc{*->}(2;0)(.5,.3){292}{0}
\rput{72}{
\psellipticarc[linestyle=dashed](2;0)(.5,.3){72}{0}
\psellipticarc{*->}(2;0)(.5,.3){0}{72}}
\rput{146}{
\psellipticarc[linestyle=dashed](2;0)(.5,.3){146}{72}
\psellipticarc{*->}(2;0)(.5,.3){72}{146}}
\rput{219}{
\psellipticarc[linestyle=dashed](2;0)(.5,.3){219}{146}
\psellipticarc{*->}(2;0)(.5,.3){146}{219}}
\rput{292}{
\psellipticarc[linestyle=dashed](2;0)(.5,.3){292}{219}
\psellipticarc{*->}(2;0)(.5,.3){219}{292}}
\NormalCoor\end{pspicture}
} \qquad{} \subfloat[Symmetries of $\mathbb{\tilde{Z}}_{n}(2)$.] {
\begin{pspicture}(-2.5,-2.5)(2.5,2.5)\SpecialCoor
\psdots[dotstyle=o](2;0)(2;72)(2;146)(2;219)(2;292)(0,0)
\psline{*->}(.5;0)(.5;146)
\psline(.5;146)(.5;292)(.5;72)(.5;219)(.5;0)
\psellipticarc[linestyle=dashed](2;0)(.5,.3){0}{292}
\psellipticarc{*->}(2;0)(.5,.3){292}{0}
\rput{72}{
\psellipticarc[linestyle=dashed](2;0)(.5,.3){146}{72}
\psellipticarc{*->}(2;0)(.5,.3){72}{146}}
\rput{146}{
\psellipticarc[linestyle=dashed](2;0)(.5,.3){292}{219}
\psellipticarc{*->}(2;0)(.5,.3){219}{292}}
\rput{219}{
\psellipticarc[linestyle=dashed](2;0)(.5,.3){72}{0}
\psellipticarc{*->}(2;0)(.5,.3){0}{72}}
\rput{292}{
\psellipticarc[linestyle=dashed](2;0)(.5,.3){219}{146}
\psellipticarc{*->}(2;0)(.5,.3){146}{219}}
\NormalCoor\end{pspicture}
}\caption{For $n=5$.}%
\end{figure}

For a general $k$, with $h=1$, the symmetries are%
\[
u_{0}(t+\zeta)=e^{-ik^{\prime}\zeta}u_{0}(t)\text{ and }u_{j+1}%
(t)=e^{\ ij\zeta}u_{1}(t+j(k\zeta)).
\]
These symmetries are similar to the ones for $k=1$, but now there is a
permutation between the phase and the planar rotation. For instance, one may
compare the cases $k=1$ and $k=2$ for $n=5$.

The effect of the group for $k=1$ was already described. In order to show the
symmetries of the group for $k=2$, notice that $k^{\prime}=3$, then the
central body visits the points: $u_{0}(\zeta)=e^{-3i\zeta}u_{0}$,
$u_{0}(2\zeta)=e^{-i\zeta}u_{0}$, $u_{0}(3\zeta)=e^{-4i\zeta}u_{0}$ and
$x_{0}(4\zeta)=e^{-2i\zeta}u_{0}$. The other $n$ bodies have orbits:
$u_{1}(t)$, $e^{\ i\zeta}u_{1}(t+2\zeta)$, $e^{\ i2\zeta}u_{1}(t+4\zeta)$,
$e^{\ i3\zeta}u_{1}(t+\zeta)$ and $e^{\ i4\zeta}u_{1}(t+3\zeta)$.

\begin{figure}[th]
\centering
\begin{pspicture}(-2.5,-2.5)(2.5,2.5)\SpecialCoor
\psdots[dotstyle=o](2;0)(2;90)(2;180)(2;270)
\psdots(0,0)
\psellipticarc[linestyle=dashed](2;0)(.5,.3){180}{0}
\psellipticarc{*->}(2;0)(.5,.3){0}{180}
\rput{90}{
\psellipticarc[linestyle=dashed](2;0)(.5,.3){0}{180}
\psellipticarc{*->}(2;0)(.5,.3){180}{0}}
\rput{180}{
\psellipticarc[linestyle=dashed](2;0)(.5,.3){180}{0}
\psellipticarc{*->}(2;0)(.5,.3){0}{180}}
\rput{270}{
\psellipticarc[linestyle=dashed](2;0)(.5,.3){0}{180}
\psellipticarc{*->}(2;0)(.5,.3){180}{0}}
\NormalCoor\end{pspicture}
\caption{Symmetries of $k=2$ for $n=4$.}%
\end{figure}
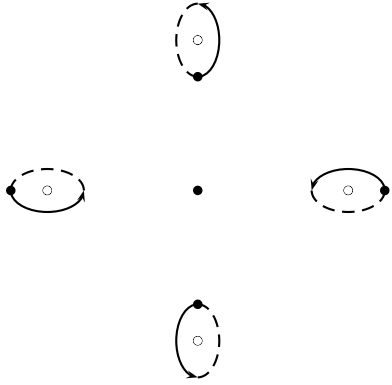

For the group with $k$ such that $h>1$, the symmetries are%
\[
u_{0}(t)=0\text{ and }u_{j+1}(t)=e^{\ ij\zeta}u_{1}(t+j\bar{k}(2\pi/\bar
{n})).
\]
In this case the central body remains at the center. Moreover, the other $n$
bodies follow the same curve with a different phase and the planar rotation
determined by multiplying by $\bar{k}$ in $\mathbb{Z}_{\bar{n}}=\{1,...,\bar
{n}\}$. See the example $n=4$ and $k=2$ .

\begin{remark}
\label{Cor}In fixed coordinates, the solutions are $q_{j}(t)=e^{i\sqrt
{\mathcal{\omega}}t}u_{j}(\nu t)$. Reparametrizing the time, we have that
$q_{j}(t)=e^{i\sqrt{\mathcal{\omega}}t/\nu}u_{j}(t)$, where $u_{j}$ is a
$2\pi$-periodic function. Set $\Omega=1-k\sqrt{\omega}/\nu$, then, for
$j\in\{1,...,n\}$, we have that%
\begin{align*}
q_{j+1}(t)  &  =e^{it\sqrt{\omega}/\nu}u_{j+1}(t)\\
&  =e^{\ ij\zeta}e^{it\sqrt{\omega}/\nu}u_{1}(t+jk\zeta)=e^{\ ij\zeta\Omega
}q_{1}(t+jk\zeta)\text{.}%
\end{align*}

In particular, when the central body has mass zero, we are considering the
$n$-body problem with equal masses. If $\Omega\in n\mathbb{Z}$, the solutions
with isotropy group $\mathbb{Z}_{n}\left(  \zeta,\zeta,-k\zeta\right)  $
satisfy
\[
q_{j+1}(t)=q_{1}(t+jk\zeta).
\]
These are solutions where all the bodies follow the same path, and they are
known as choreographies, \cite{Ch08}.
\end{remark}

\subsection{Spatial solutions}

Now we wish to describe the symmetries of a solution with isotropy group
$\mathbb{Z}_{n}\left(  \zeta,\zeta,-k\zeta\right)  \times\mathbb{Z}_{2}%
(\kappa,\pi)$. As we saw for the group $\mathbb{Z}_{2}$, these solutions
satisfy%
\[
u_{j}(t)=u_{j}(t+\pi)\text{ and }z_{j}(t)=-z_{j}(t+\pi)\text{,}%
\]
where $u_{j}=x_{j}+iy_{j}$ is the projection in the $(x,y)$-plane.

Since the group $\mathbb{Z}_{n}$ is generated by $(\zeta,\zeta,-k\zeta)$,
these solutions satisfy the symmetries $u_{j}(t)=e^{-i\zeta}u_{\zeta
(j)}(t-k\zeta)$ and
\begin{equation}
z_{j}(t)=z_{\zeta(j)}(t-k\zeta)\text{.} \label{ss}%
\end{equation}
Therefore, the projection in the $(x,y)$-plane follows a $\pi$-periodic curve
with the symmetries of the previous section. The spatial solutions go around
this projected curve twice, once with $z_{j}(t)$ and once with $z_{j}(-t)$.
Hence, the solutions look like spatial eight where the spatial positions
satisfy the symmetries (\ref{ss}).

Hence, we need to identify the spatial symmetries (\ref{ss}). For the central
body we have that
\[
z_{0}(t)=z_{0}(t+k\zeta).
\]
Using the notation $z_{j}=z_{j+kn}$ for the other $n$ bodies, their spatial
positions are related by%
\[
z_{j+1}(t)=z_{1}(t+jk\zeta).
\]

Therefore, the spatial curves of the $n$ bodies are determined by just one of
them. To see one example, we suppose that $n=2m$ and choose $k=m$. In this
case the central body remains at the center. Moreover, the $n$ bodies with
equal masses satisfy
\[
u_{j+1}(t)=e^{\ ij\zeta}u_{1}(t+j\pi)=e^{\ ij\zeta}u_{1}(t)
\]
and%
\[
z_{j+1}(t)=z_{1}(t+j\pi)=(-1)^{j}z_{1}(t).
\]
Thus, there are two $m$-polygons which oscillate vertically, one with
$z_{1}(t)$ and the other with $-z_{1}(t)$. Furthermore, the projection of the
two $m$-polygons in the plane is always a $2m$-polygon. These solutions are
known as Hip-Hop orbits.

\begin{remark}
An element in the intersection of two isotropy subspaces, for $k_{1}$ and
$k_{2}$, will be such that $u_{0}(t)=0$, $u_{j}(t)$, $z_{0}(t)$, $z_{j}(t)$
will be $(k_{2}-k_{1})\zeta$-periodic, that is, if $h$ is the greatest common
factor of $n$ and $k_{2}-k_{1}$ and $\tilde{n}h=n$, then the period will be
$2\pi/\tilde{n}$.
\end{remark}

\section{The bifurcation theorem}

The orthogonal degree is defined for orthogonal maps which are non-zero on the
boundary of some open bounded invariant set. The degree is made of integers,
one for each orbit type, and it has all the properties of the usual Brouwer
degree. Hence, if one of the integers is non-zero, then the map has a zero
corresponding to the orbit type of that integer. In addition, the degree is
invariant under orthogonal deformations that are non-zero on the boundary. The
degree has other properties such as sum, products and suspensions, for
instance, the degree of two pieces of the set is the sum of the degrees. The
interested reader may consult \cite{IzVi03}, Chapters 2 and 4,\cite{BaKrSt06}
and \cite{Ry05} for more details on equivariant degree and degree for gradient maps.

Now, if one has an isolated orbit, then its linearization at one point of the
orbit $x_{0}$ has a block diagonal structure, due to Schur's lemma,
\cite{IzVi03}, Lemma 7.2, p.30, where the isotropy subgroup of $x_{0}$ acts as
$\mathbb{Z}_{n}$ or as $S^{1}$. Therefore, the orthogonal index of the orbit
is given by the signs of the determinants of the submatrices where the action
is as $\mathbb{Z}_{n}$, for $n=1$ and $n=2$, and the Morse indices of the
submatrices where the action is as $S^{1}$, \cite{IzVi03}, Theorem 3.1, p.
247. In particular, for problems with a parameter, if the orthogonal index
changes at some value of the parameter, one will have bifurcation of solutions
with the corresponding orbit type. Here, the parameter is the frequency $\nu$,
\cite{IzVi03}, Proposition 3.1, p. 255.

For a $k$-dimensional orbit with a tangent space generated by $k$ of the
infinitesimal generators of the action of the group, one uses a Poincar\'{e}
section for the map augmented with $k$ Lagrange-like multipliers for the
generators. (See the construction in \cite{IzVi03}, Section 4.3, p. 245.) For
instance, for the action of $SO(2)$, the study of the zeros of the equivariant
map $F(x)$, orthogonal to the generator $Ax$, is equivalent to the study of
the zeros of $F(x)+\lambda Ax$, if $x$ is not fixed by the the group, i.e., if
$Ax$ is not $0$, for which $\lambda$ is $0$. In this way, one has added an
artificial parameter. This trick has been used very often and, in the context
of a topological degree argument, was called \textquotedblleft orthogonal
degree\textquotedblright\ by S. Rybicki in \cite{Ry94}. See also \cite{Da85}
and \cite{IMV89}, p. 481, for the case of gradients. The general case of the
action of an abelian group was treated in \cite{IzVi99}. The complete study of
the orthogonal degree theory is given in \cite{IzVi03}, Chapters 2 and 4.

Any Fourier mode will give rise to an orbit type (modes which are multiples of
it), hence one has an element of the orthogonal degree for each mode.
Furthermore, if $x(t)$ is a periodic solution, with frequency $\nu$, then
$y(t)= x(nt)$ is a $2\pi/n$-periodic solution, with frequency $\nu/n$. Hence,
any branch arising from the fundamental mode will be reproduced in the
harmonic branch. If one wishes to study period-doubling, then one has to
consider the branch corresponding to $\pi$-periodic solutions, \cite{IzVi03},
Corollary 2.1, p.219.

The complete study of the orthogonal degree theory is given in \cite{IzVi03},
Chapters 2 and 4.

\subsection{General relative equilibrium}

To use successfully the results of the orthogonal degree from \cite{IzVi03},
Proposition 3.2, p. 258, we need to verify the hypothesis that the orbit
$\Gamma x$ is hyperbolic near a bifurcation point, \cite{IzVi03}, Definition
2.2, p. 222. In particular, we need to prove that the kernel of $M(0)^{\perp
e}$ is generated by $A_{1}x_{0}=-\mathcal{\bar{J}}x_{0}$. Since $A_{1}x_{0}\in
V_{0}$ and $e\in V_{1}$, an equivalent condition consists in showing that the
kernel of $M_{0}(0)$ is generated by $T_{0}^{-1}(A_{1}x_{0}) $, and that the
kernel of $M_{1}(0)$ is generated by $T^{-1}(e)=(1,...,1)$.

\begin{definition}
Following \cite{IzVi03}, p. 258, we define $\sigma$ as the sign of the
determinant of $M_{0}(0)$ in the space orthogonal to $T_{0}^{-1}%
(\mathcal{\bar{J}}x_{0})$.
\end{definition}

Now, if $(z_{1},...,z_{n})$ is in the kernel of $M_{1}(0)=(a_{ij})$, then
$z_{j}=z_{i}$, because%
\[
0=\sum_{j=1}^{n}a_{ij}z_{j}=\sum_{j=1~(j\neq i)}^{n}(z_{j}-z_{i})a_{ij}.
\]
Consequently, the kernel of $M_{1}(0)$ is generated by $(1,...,1)$, and the
hypothesis of \cite{IzVi03}, Proposition 3.2, p. 258, near a bifurcation
point, is assured by the condition $\sigma\neq0$.

\begin{definition}
Let $n_{j}(\nu)$ be the Morse number of $M_{j}(\nu)$, for $j=0,1$, and define%
\[
\eta_{j}(\nu)=\sigma(n_{j}(\nu-\rho)-n_{j}(\nu+\rho)).
\]

\end{definition}

The number $\eta_{j}(\nu_{0})$ represents the change of the Morse index at the
point $\nu_{0}$. Applying the bifurcation theorems of (\cite{IzVi03}), Remark
3.5, p. 259, we get the following result:

\begin{theorem}
If $\eta_{j}(\nu_{j})$ is different from zero, the equilibrium $x_{0}$ has a
global bifurcation of periodic solutions from $2\pi/\nu_{j}$, with isotropy
group $\mathbb{Z}_{2}(\kappa,j\pi)$.
\end{theorem}

The bifurcation branch is \emph{non-admissible} when a) the norm or the period
goes to infinity or b) the branch ends in a collision path. In any other case,
we say that the bifurcation is \emph{admissible. }By global bifurcation we
mean that: if the branch is admissible, then the branch must return to other
bifurcation points, and the sum of the local degrees at the bifurcation points
$\eta_{k}(\nu)$ must be zero. See \cite{Iz95} for this kind of arguments.

\subsection{The polygonal equilibria}

We have decomposed the matrix $M(\nu)$ into the two blocks $M_{0}(\nu)$ and
$M_{1}(\nu)$. Thus, there is a bifurcation of periodic solutions when the
Morse index of $M_{k}(\nu)$ changes.

For the polygonal equilibrium, we have decomposed the blocks $M_{0}(\nu)$ and
$M_{1}(\nu)$ into $m_{01}(\nu),...,m_{0n}(\nu)$ and $m_{11}(\nu),...,m_{1n}%
(\nu)$ respectively. Actually, the orthogonal degree for the polygonal
equilibrium has one element for each matrix $m_{jk}(\nu)$. Thus, we may prove
a more complete result because the group of symmetries is bigger.

\begin{definition}
For $j=0,1$ and $k=1,...,n$, set $n_{jk}(\nu)$ to be the Morse number of
$m_{jk}(\nu)$, and define%
\[
\eta_{jk}(\nu)=\sigma(n_{jk}(\nu-\rho)-n_{jk}(\nu+\rho))\text{.}%
\]

\end{definition}

The fixed point space of the isotropy group $\Gamma_{\bar{a}}\times S^{1}$
corresponds to the block $m_{0n}(0)=B_{n}$. Since the generator of the kernel
is $A_{1}\bar{a}=T_{n}(-n^{1/2}e_{2})$, then, in this case, $e_{2}$ is in the
kernel of $m_{0n}(0)$.

Let $\sigma$ be the sign of $m_{0n}(0)$ in the space orthogonal to $e_{2}$.
Since $\beta_{n}=-\alpha_{+}s_{1}$ and $\alpha_{n}=\alpha_{-}s_{1}$, then
$B_{n}=diag((\alpha+1)(\mu+s_{1}),0)$. Thus, for the polygonal equilibrium,
$\sigma$ is the sign of $(\alpha+1)(\mu+s_{1})$, \cite{GaIz11}, p. 3220,%
\[
\sigma=\sigma_{n}(\mu)=1\text{.}%
\]

In this case, we still need the hyperbolic condition near a bifurcation point.
That is, we need to be sure that the matrices $m_{0k}(0)=B_{k}$ are invertible
for $k=1,...,n-1$. In the paper \cite{GaIz11}, Proposition 23, we did prove
that the block $B_{k}$ is invertible except for one point $\mu_{k}\in
(s_{1},\infty)$. In fact, there is a bifurcation branch of relative equilibria
from each $\mu_{k}$ for $k=1,...,n-1$.

From the results of \cite{IzVi03}, Remark 3.5, p. 259, we may state the
following theorem for $\mu\neq\mu_{1},...,\mu_{n-1}$.

\begin{theorem}
If $\eta_{jk}(\nu_{k})$ is different from zero, then the polygonal equilibrium
has a global bifurcation of periodic solutions from $2\pi/\nu_{k}$, with
isotropy group $\mathbb{Z}_{n}\left(  \zeta,\zeta,-k\zeta\right)
\times\mathbb{Z}_{2}(\kappa,j\pi)$.
\end{theorem}

\section{Spectral analysis}

In the previous section, we have proved that there is a bifurcation of
periodic solutions when the blocks of $M_{k}(\nu)$ change their Morse index.
The spatial and planar blocks were given in proposition (\ref{E3.1.0}). We
will analyze the spatial spectrum for a general relative equilibrium,
corresponding to $M_{1}(\nu)$.

For the polygonal equilibrium, the blocks are given in propositions
(\ref{E3.2.0}) and (\ref{E3.2.1}), and the description of the isotropy groups
was done in section four. In this section we analyze completely the spectrum
of all these blocks $m_{jk}(\nu)$.

\subsection{Spatial spectrum for\ a general equilibrium}

The spatial block of a general relative equilibrium is
\[
M_{1}(\nu)=\nu^{2}\mathcal{M}_{1}+(a_{ij})_{ij=1}^{n}\text{,}%
\]
were the matrix $(a_{ij})$ was given in proposition (\ref{E1.1.1})

The following result is well known in linear algebra.

\begin{lemma}
Let $B(z_{0},r)$ be the ball of center $z_{0}$ and radius $r$. Then, the
spectrum $\sigma(A)$ of a matrix $A=(a_{ij})_{ij=1}^{n}$ is in the union
$\bigcup_{i=1}^{n}B(a_{ii},r_{i})$ with $r_{i}=\sum_{j=1(j\neq i)}%
^{n}\left\vert a_{ij}\right\vert $.
\end{lemma}

The matrix $(a_{ij})$ has real eigenvalues because it is selfadjoint. And from
the previous lemma, we have that the eigenvalues of the matrix $(a_{ij}) $ are
in the union $\bigcup_{i=1}^{n}B(a_{ii},\left\vert a_{ii}\right\vert )$,
because $a_{ii}=-\sum_{j\neq i}a_{ij}.$ Since $a_{ii}$ is negative, then
$M_{1}(0)=(a_{ij})$ must have $n-1$ negative eigenvalues.

Since the eigenvalues of $M_{1}(\nu)$ are continuous, and $M_{1}(0)$ has $n-1
$ negative eigenvalues, then the Morse number of $M_{1}(\nu)$ satisfies
$n_{1}(\nu)\geq n-1$ for any small $\nu$. Now, since $n_{1}(\infty)=0$, then
the matrix $M_{1}(\nu)$ must change its Morse index at $n-1$ values of $\nu$.
However, these $n-1$ values could be the same, and we may assure the existence
of only one point where the Morse index changes.

\begin{theorem}
For the $n$-body problem, each relative equilibrium with $\sigma\neq0$ has at
least one global bifurcation branch of periodic solutions with symmetries
$\mathbb{Z}_{2}(\kappa,\pi)$. Generically, the equilibrium has $n-1$
bifurcations of periodic eight solutions.
\end{theorem}

In the particular case where all the bodies have the same mass, $m_{j}=m$, let
$P$ be a matrix such that $m^{-1}(a_{ij})=P^{-1}\Lambda P$, where
$\Lambda=diag(-\nu_{1}^{2},...,-\nu_{n}^{2})$ with%
\[
-\nu_{1}^{2}\leq...\leq-\nu_{n-1}^{2}<\nu_{n}=0\text{.}%
\]
Since%
\[
M_{1}(\nu)=mP^{-1}(\nu^{2}I+\Lambda)P\text{,}%
\]
then the matrix $M_{1}(\nu)$ has eigenvalues $\lambda_{k}=m(\nu^{2}-\nu
_{k}^{2})$.

If all the $\nu_{k}$'s are different, then $n_{1}(\nu_{k}-\rho)=k$ and
$n_{1}(\nu_{k}+\rho)=k-1$. Thus, for $k=1,...,n-1$, the change of Morse index
of $M_{1}(\nu)$ at the points $\nu_{k}$ is%
\[
\eta_{1}(\nu_{k})=\sigma.
\]

If the eigenvalue $\nu_{k}$ has multiplicity $j$, it is easy to prove that the
change of the Morse index is $\eta_{1}(\nu_{k})=j\sigma$. In any case, all
bifurcation points have indices of the same sign, then for the equal mass
$n$-body problem, these bifurcations branches must be non-admissible or
connect to other relative equilibria.

\subsection{Spatial spectrum for the polygonal equilibrium}

\begin{theorem}
Define $\nu_{k}=\sqrt{\mu+s_{k}}$ for $k\in\{1,...,n-1\}$, and $\nu_{n}%
=\sqrt{\mu+n}$. Then, the polygonal equilibrium has a global bifurcation of
periodic solutions from $2\pi/\nu_{k}$ with symmetries\ $\mathbb{Z}_{n}\left(
\zeta,\zeta,-k\zeta\right)  \times\mathbb{Z}_{2}(\kappa,\pi)$ for each
$k\in\{1,...,n\}$.
\end{theorem}

\begin{proof}
For $k\in\{1,...,n-1\}$, we have that $m_{1k}(\nu)=\nu^{2}-(\mu+s_{k})$
changes only at $\nu_{k}$. Since $\sigma=1$, then $\eta_{1}(\nu_{k})=1$. For
$k=n$, we have that
\[
\det m_{1n}(\nu)=\nu^{2}\mu\left(  \nu^{2}-(\mu+n)\right)
\]
changes only at $\nu_{n}$. Since $\det m_{1n}$ is negative for $\nu<\nu_{n}$,
then $n_{1n}(\nu_{n}-\rho)=1$. Moreover, since $n_{1n}(\infty)=0$, then
$\eta_{1n}(\nu_{n})=1$. The existence of the bifurcations follows from the
bifurcation theorem.
\end{proof}

Since the spatial bifurcations for the polygonal equilibrium have index
$\eta=1$, then all the bifurcation branches must be non-admissible or go to
the other equilibria.

\subsection{Planar spectrum for the polygonal equilibrium}

In the polygonal equilibrium $\bar{a}$, the frequency is $\omega=\mu+s_{1}$.
The equations have physical meaning for positive mass, $\mu>0$, although, the
problem is well defined for $\mu>-s_{1}$. In this section we will analyze the
spectrum of the blocks $m_{0k}(\nu)$, for $\mu>-s_{1}$.

Normalize the period of $m_{0k}(\nu)$ by $\sqrt{\omega}$ and define $m_{k}%
(\nu)=m_{0k}(\sqrt{\omega}\nu)$, then%
\begin{align*}
m_{k}(\nu)  &  =\omega\lbrack\nu^{2}I-2\nu(iJ)]+B_{k}\text{ for }%
k\in\{2,...,n-2,n\}\text{ and}\\
m_{1}(\nu)  &  =\omega\lbrack\nu^{2}diag(\mu,I)-2\nu diag(\mu,iJ)]+B_{1}%
\text{.}%
\end{align*}

Notice that $m_{k}(\nu)$ is selfadjoint, so with real eigenvalues. Since
$m_{n-k}(\nu)=\bar{m}_{k}(-\nu)$, then the matrices $m_{n-k}(\nu)$ and
$m_{k}(-\nu)$ have the same spectrum. Thus, the Morse numbers $n_{k}(\nu)$
satisfy
\[
n_{n-k}(\nu)=n_{k}(-\nu).
\]

We cannot calculate explicitly the sums $s_{k}$, but we shall use that the
sums $s_{k}$ are positive, satisfy $s_{k}=s_{n-k}=s_{n+k}$, and that $s_{k}$
are increasing in $k$ for $k\in\{0,...,n/2\}$. This last fact is proved in the
appendix. The matrices $B_{k}$ are given in the proposition (\ref{E3.2.0}).
However, in order to simplify the computations, we shall restrict the analysis
to the Newton case, that is with $\alpha=2$.

\subsubsection{Block $k=n$}

The block $B_{n}$ is $B_{n}=(3/2)(\mu+s_{1})(I+R)$ with $R=diag(1,-1)$.
Therefore,%
\[
\sigma=sgn(e_{1}^{T}B_{n}e_{1})=1,
\]
for $\mu>-s_{1}$.

\begin{proposition}
The matrix $m_{n}(\nu)$ changes its Morse index only at the positive value
$\nu=1$ with
\[
\eta_{n}(1)=1.
\]

\end{proposition}

\begin{proof}
The block $m_{n}(\nu)$ is
\[
m_{n}(\nu)=\omega\lbrack\nu^{2}-2\nu(iJ)+diag(3,0)]\text{.}%
\]
Thus, the determinant $d_{n}(\nu)=\omega^{2}\nu^{2}\left(  \nu-1\right)
\left(  \nu+1\right)  $ is zero only at $\pm1$. Since $d_{n}(\varepsilon)<0$,
then $n_{n}(\varepsilon)=1$. Moreover, since $n_{n}(\infty)=0$, then $\eta
_{n}(1)=1$.
\end{proof}

Therefore, there is a global branch of periodic solutions bifurcating from the
polygonal equilibrium starting with the period $2\pi$ and with symmetries
$\mathbb{Z}_{n}\left(  \zeta,\zeta,-k\zeta\right)  \times\mathbb{Z}_{2}%
(\kappa)$.

\begin{remark}
In fixed coordinates, the central body satisfies $q_{0}(t)=0$ and the other
$n$ bodies satisfy $u_{j+1}(t)=e^{ij\zeta}u_{1}(t)$. Now, if $\nu$ remains $1$
on the branch, that is, if the frequency, of the solutions in fixed
coordinates, is $\omega$, then $q_{j+1}(t)= e^{ij\zeta}q_{1}(t)$. Thus, this
bifurcation branch may be made of solutions moving on curves looking like ellipses.
\end{remark}

\subsubsection{Blocks $k\in\{2,...,n-2\}$}

Define $d_{k}(\nu)$ to be the determinant of $m_{k}(\nu)$.

\begin{proposition}
The determinant of $m_{k}(\nu)$ is
\[
d_{k}(\nu)=\omega^{2}\nu^{4}+(2\alpha_{k}-\omega-s_{1})\omega\nu^{2}%
-4\omega\gamma_{k}\nu+a_{k}+\mu b_{k}\text{,}%
\]
where
\[
a_{k}=(s_{1}+\alpha_{k})^{2}-\beta_{k}^{2}-\gamma_{k}^{2}\text{ and }%
b_{k}=3(s_{1}+\alpha_{k}+\beta_{k})>0\text{.}%
\]

\end{proposition}

\begin{proof}
The block $m_{k}(\nu)$ is
\[
m_{k}(\nu)=(\omega\nu^{2}+s_{1}+\alpha_{k})I+(3\mu/2)(I+R)-\beta_{k}%
R-(2\nu\omega+\gamma_{k})(iJ).
\]
Therefore, the determinant is%
\begin{align*}
d_{k}(\nu)  &  =(\omega\nu^{2}+s_{1}+\alpha_{k}-\beta_{k}+3\mu)(\omega\nu
^{2}+s_{1}+\alpha_{k}+\beta_{k})\\
&  -(2\omega\nu+\gamma_{k})^{2}\text{.}%
\end{align*}
Since $\omega=\mu+s_{1}$ and $d_{k}(0)=b_{k}\mu+a_{k}$, then%
\[
d_{k}(\nu)=\omega^{2}\nu^{4}+(2\alpha_{k}-\omega-s_{1})\omega\nu^{2}%
-4\omega\gamma_{k}\nu+a_{k}+\mu b_{k}\text{.}%
\]

\end{proof}

We write the determinant in terms of $\omega$ as
\[
d_{k}(\nu)=a\omega^{2}+b\omega-c\text{,}%
\]
where
\[
a=\nu^{2}(\nu^{2}-1)\text{, }b=\nu^{2}(2\alpha_{k}-s_{1})-4\nu\gamma_{k}%
+b_{k}\text{ and }c=s_{1}b_{k}-a_{k}\text{.}%
\]
As a consequence of the next lemma, we shall prove that $b^{2}+4ac$ is
positive. Then $d_{k}(\mu,\nu)$ is zero at exactly the two real solutions
\[
\omega_{\pm}(\nu)=\frac{-b\pm\sqrt{b^{2}+4ac}}{2a}\text{.}%
\]

\begin{lemma}
For any $\nu\in\mathbb{R}$, we have that $c$ and $b(\nu)$ are positive and
satisfy%
\[
b^{2}(\nu)>9c\text{.}%
\]

\end{lemma}

\begin{proof}
From the definitions of $a_{k}$ and $b_{k}$ we have%
\[
4s_{1}b_{k}=3s_{1}s_{k-1}+3s_{1}s_{k+1}-6s_{1}^{2}+18s_{1}s_{k}\text{,}%
\]
and%
\[
4a_{k}=2s_{1}s_{k-1}+2s_{1}s_{k+1}-5s_{1}^{2}-9s_{k}^{2}+18s_{1}s_{k}%
+s_{k-1}s_{k+1}\text{.}%
\]
Therefore,%
\[
4c=4(s_{1}b_{k}-a_{k})=9s_{k}^{2}-\left(  s_{k-1}-s_{1}\right)  \left(
s_{k+1}-s_{1}\right)  \text{.}%
\]
Using the inequality of the appendix
\[
(s_{k+1}-s_{1})<2s_{k}-(s_{k-1}-s_{1})<2s_{k},
\]
we conclude that
\[
7s_{k}^{2}<4c\leq9s_{k}^{2}.
\]

Since $b^{\prime}(\nu)=2\nu(2\alpha_{k}-s_{1})-4\gamma_{k}$, then $b^{\prime
}(\nu)=0$ at $\nu_{0}=2\gamma_{k}/(2\alpha_{k}-s_{1})$. Since $2\alpha
_{k}-s_{1}>0$, then $\nu_{0}$ is a minimum of $b(\nu)$. Moreover, since
$(2\alpha_{k}-s_{1})-2\gamma_{k}=s_{k-1}-s_{1}\geq0$, then $\nu_{0}\in(0,1)$.
Thus,
\[
b(\nu_{0})=\nu_{0}(2\gamma_{k}-4\gamma_{k})+b_{k}>b_{k}-2\gamma_{k}.
\]
Therefore,%
\[
4b(\nu_{0})>4b_{k}-8\gamma_{k}=18s_{k}-6s_{1}+5s_{k-1}+s_{k+1}>18s_{k}\text{.}%
\]
From these inequalities one has that $b^{2}(\nu)>9(9s_{k}^{2}/4)>9c.$
\end{proof}

For $\nu^{2}>1$, since $a$ is positive, then $b^{2}+4ac$ is positive. For
$\nu^{2}\in(0,1)$, since $a\in(-1/4,0)$ and $b^{2}(\nu)>4c$, then
$b^{2}+4ac\geq4c(1+a)\geq0$. Consequently, the two solutions $\omega_{\pm}%
(\nu)\ $are real for $\nu\neq-1,0,1$.

Remember that the equations of the bodies have meaning only for $\omega
=\mu+s_{1}$ $>0$. For $\nu^{2}>1$, since $4ac$ is positive, then $\omega
_{+}(\nu)$ is positive and $\omega_{-}(\nu)$ is negative. For $\nu^{2}%
\in(0,1)$, since $a$ and $4ac$ are negative, then the two roots $\omega_{\pm
}(\nu)$ have the sign of $b$. Since $b$ is positive, then the two solutions
$\omega_{\pm}(\nu)$ are positive for $\nu^{2}\in(0,1)$.

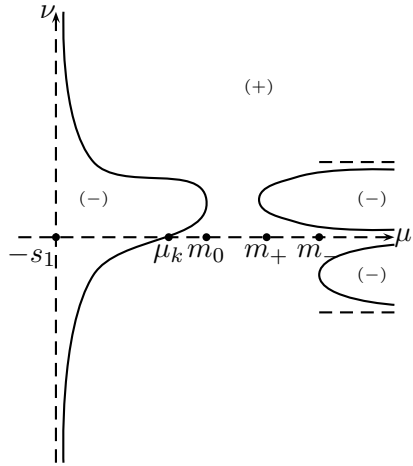
\begin{figure}[th]
\centering
\subfloat
{\ \begin{pspicture}(-3,-3)(3,3)\SpecialCoor
\psline[linestyle=dashed]{->}(-1.5,-3)(-1.5,3)
\psline[linestyle=dashed]{->}(-2,0)(3,0)
\psline[linestyle=dashed](2,1)(3,1)
\psline[linestyle=dashed](2,-1)(3,-1)
\pscurve(-1.4,3)(-1,1)(.5,.5)(-1,-.5)(-1.4,-3)
\pscurve(3,.9)(1.7,.8)(1.2,.5)(1.7,.2)(3,.1)
\pscurve(3,-.9)(2,-.5)(3,-.1)
\rput[l](3,0){\small $\mu$}
\rput[r](-1.5,3){\small $\nu$}
\rput[l](1,2){\tiny $(+)$}
\rput[l](-1.2,.5){\tiny $(-)$}
\rput[l](2.5,.5){\tiny $(-)$}
\rput[l](2.5,-.5){\tiny $(-)$}
\psdots[dotsize=3pt](-1.5,0)(0,0)(.5,0)(1.3,0)(2,0)
\rput[tr](-1.5,-.1){\small $-s_{1}$}
\rput[t](0,-.1){\small $\mu_{k}$}
\rput[t](.5,-.1){\small $m_{0}$}
\rput[t](1.3,-.1){\small $m_{+}$}
\rput[t](2,-.1){\small $m_{-}$}
\NormalCoor\end{pspicture}
}\caption{Graph of $d_{k}(\mu,\nu)=0$.}%
\end{figure}

Defining $\mu_{\pm}(\nu)=\omega_{\pm}(\nu)-s_{1}$, the matrix $m_{k}(\mu,\nu)
$ changes its Morse index only at the two curves%
\[
\mu_{+}\mathbb{(\nu)}\text{ for }\nu\in\mathbb{R}\text{, and }\mu_{-}%
(\nu)\text{ for }\nu\in(-1,0)\cup(0,1)\text{.}%
\]
These curves satisfy the inequalities $-s_{1}<\mu_{+}(\nu)<\mu_{-}(\nu)$ for
$\left\vert \nu\right\vert \in(0,1)$. Moreover, since $a\rightarrow0$ when
$\left\vert \nu\right\vert \rightarrow\{0,1\}$, then $\mu_{-}(\nu
)\rightarrow\infty$ when $\left\vert \nu\right\vert \rightarrow\{0,1\}$.
Furthermore, since $\omega_{+}(\nu)\rightarrow0$ when $\left\vert
\nu\right\vert \rightarrow\infty$, then $\mu_{+}(\nu)\rightarrow-s_{1}$ when
$\left\vert \nu\right\vert \rightarrow\infty$.

\begin{definition}
Let $m_{0}$ be the maximum of $\mu_{+}(\nu)$ in $\mathbb{R}$, $m_{+}$ be the
minimum of $\mu_{-}(\nu)$ on $(0,1)$, and $m_{-}$ be the minimum of $\mu
_{-}(\nu)$ on $(-1,0)$.
\end{definition}

\begin{proposition}
For $k\in\{2,...,n-2\}$, if $\mu\in$ $(-s_{1},\mu_{k})$, then $m_{k}(\nu)$
changes its Morse index at a positive value $\nu_{k}$, with
\[
\eta_{k}(\nu_{k})=1.
\]

For $k\in\{2,...,[n/2]\}$, if $\mu\in(\mu_{k},m_{0})$, then $m_{k}(\nu)$
changes its Morse index at two positive values $\bar{\nu}_{k}<\nu_{k}$, with%
\[
\eta_{k}(\bar{\nu}_{k})=-1\text{ and }\eta_{k}(\nu_{k})=1.
\]

For $k\in\{2,...,n-2\}$, if $\mu\in(m_{+},\infty)$, then $m_{k}(\nu)$ changes
its Morse index at two positive values $\bar{\nu}_{k}<\nu_{k}<1$, with%
\[
\eta_{k}(\bar{\nu}_{k})=-1\text{ and }\eta_{k}(\nu_{k})=1.
\]

\end{proposition}

\begin{proof}
Since $\mu_{k}=-b_{k}/a_{k}$ is the only zero of $d_{k}(\mu,0)$, then $\mu
_{k}\leq m_{0}$. Notice that the function $d_{k}(\nu)+4\omega\gamma_{k}\nu$ is
even in $\nu$, then%
\[
d_{k}(\mu,\nu)=d_{k}(\mu,-\nu)-8\omega\gamma_{k}\nu.
\]
Since $\gamma_{k}>0$ for $k\in\lbrack2,...,n/2)\cap\mathbb{N}$, then
$d_{k}(\mu,\nu)<d_{k}(\mu,-\nu)$ for $\nu\in\mathbb{R}^{+}$. Thus, the maximum
$m_{0}$ is reached for some $\nu>0$, and $m_{+}<m_{-}$.

Since $m_{k}(\nu)$ is a $2\times2$ matrix, then the Morse number is $n_{k}%
(\nu)=1$ when $d_{k}(\nu)$ is negative. Hence, the Morse index is $n_{k}%
(\mu,\nu)=1$ in $\Omega$, where
\[
\Omega=\{(\mu,\nu):\mu<\mu_{0}(\nu),\mu>\mu_{\pm}(\nu)\}.
\]
Since $n_{k}(\nu)=0$ for $\nu$ large enough, and since $m_{k}(\nu)$ changes
its Morse index only at $\partial\Omega$, then $n_{1}(\mu,\nu)=0$ in
$\bar{\Omega}^{c}$.

As we have seen before, $\sigma=1$ for $\mu>-s_{1}$, then, for the first case,
we have that $\eta_{k}(\nu_{k})=1-0$. For the other two cases, there are two
values, $\bar{\nu}_{k}<\nu_{k}$, with $\eta_{k}(\bar{\nu}_{k})=0-1$ and
$\eta_{k}(\nu_{k})=1-0$.
\end{proof}

\begin{remark}
Since
\[
d_{k}^{\prime\prime}(\nu)=2(6\omega^{2}\nu^{2}+(2\alpha_{k}-\mu-2s_{1}%
)\omega),
\]
then $d_{k}^{\prime\prime}(\nu)$\ is positive for $\mu<2(\alpha_{k}-s_{1})$.
Thus, for $\mu\in(-s_{1},2(\alpha_{k}-s_{1}))$ the matrix $m_{k}(\nu)$ changes
its Morse index only at $\nu_{k}$, for the first case. Now, since $d_{k}(\nu)$
is a polynomial in $\nu$ of degree $4$, then $d_{k}(\nu)$ has at most four
zeros. This is the case for $\mu>\max\{m_{+},m_{-}\}$, and then the two
positive solutions for the third case are the only ones.
\end{remark}

From the bifurcation theorem we have the following result, for $n\geq4$.

\begin{theorem}
\label{E6.3.1}The polygonal equilibrium has a global bifurcation of planar
periodic solutions with symmetries $\mathbb{Z}_{n}\left(  \zeta,\zeta
,-k\zeta\right)  $ for each $k\in\{2,...,n-2\}$ and $\mu\in(-s_{1},\mu_{k})$.
When $\mu\in(-s_{1},2(\alpha_{k}-s_{1}))$, these bifurcation branches are
non-admissible or go to another equilibrium .

For each $k\in\{2,...,[n/2]\}$ and $\mu\in(\mu_{k},m_{0})$, and for each
$k\in\{2,...,n-2\}$ and $\mu\in(m_{+},\infty)$, the polygonal equilibrium has
two global bifurcations of planar periodic solutions with symmetries
$\mathbb{Z}_{n}\left(  \zeta,\zeta,-k\zeta\right)  $.
\end{theorem}

\subsubsection{Blocks $k\in\{1,n-1\}$}

From the definition, the block $m_{1}(\nu)$ is%
\[
\left(
\begin{array}
[c]{ccc}%
\mu\left(  \omega\nu^{2}-2\nu\omega+s_{1}+\mu+n/2\right)  & -2\left(
n/2\right)  ^{1/2}\mu & -\left(  n/2\right)  ^{1/2}\mu i\\
-2\left(  n/2\right)  ^{1/2}\mu & \omega\nu^{2}+s_{1}+\alpha_{1}+3\mu &
\alpha_{1}i+2\nu\omega i\\
\left(  n/2\right)  ^{1/2}\mu i & -\alpha_{1}i-2\nu\omega i & \omega\nu
^{2}+s_{1}+\alpha_{1}%
\end{array}
\right)  \text{.}%
\]
Let us define $a_{1}$ and $b_{1}$ as%
\[
a_{1}=(2s_{1}+n)(2s_{1}+s_{2})/4\text{ and }b_{1}=3(4s_{1}+s_{2}-2n)/4.
\]
The determinant of $m_{1}(\nu)$ may be written as%
\[
\det m_{1}(\nu)=\omega\mu(\nu-1)^{2}d_{1}(\nu),
\]
where $d_{1}(\nu)$ is the polynomial
\[
d_{1}(\nu)=\nu^{2}(\nu^{2}-1)\omega^{2}+\left[  \left(  s_{2}-2s_{1}+n\right)
\nu^{2}/2-(s_{2}-n)\nu\right]  \omega+(a_{1}+\mu b_{1})\text{.}%
\]

Therefore, we need to find the points where $d_{1}(\nu)$ is zero.

\begin{lemma}
\label{5EnLe}The functions $b_{1}$, $2\alpha_{1}-s_{1}$, $s_{1}-n$, $s_{2}-n $
and $8(s_{1}-n)+9s_{2}$ are positive, at least for $n>1071$. For $n\geq3$, we
get numerically that $s_{1}-n$ is negative if $n\leq472$ and positive if
$n\geq473$, that $s_{2}-n$ is negative if $n\leq11$ and positive if $n\geq12$,
and
\[
sgn(b_{1})=sgn(2\alpha_{1}-s_{1})=sgn(8(s_{1}-n)+9s_{2})=\left\{
\begin{array}
[c]{c}%
-1\text{ for }n\leq6\\
+1\text{ for }n\geq7
\end{array}
\right.  \text{.}%
\]

\end{lemma}

\begin{proof}
Since $\sin x<x$, then%
\[
s_{1}=\frac{1}{4}\sum_{j=1}^{n-1}\frac{1}{\sin j\zeta/2}\geq\frac{1}{2}%
\sum_{j\in\lbrack1,n/2)\cap\mathbb{N}}\frac{1}{j\zeta/2}\geq\frac{1}{\zeta}%
\ln(n/2)=n\frac{\ln(n/2)}{2\pi}.
\]
Since $\ln(n/2)\geq2\pi$ for $n>1071$, then $s_{1}>n$ at least for $n>1071 $.

Using that $s_{2}>s_{1}$ and $b_{1}=4(2(s_{1}-n)+s_{2}+2s_{1})/4$, we conclude
that $b_{1}$, $s_{2}-n$ and $8(s_{1}-n)+9s_{2}$ are positive for $n>1071$.
Finally, from proposition (\ref{EA.1.0}), we have that $s_{2}=4s_{1}-\bar
{s}_{1}$ with $\bar{s}_{1}=\sum_{j=1}^{n-1}\sin(j\zeta/2)<n $, then,
\[
2\alpha_{1}-s_{1}=s_{2}/2-s_{1}=s_{1}-\bar{s}_{1}/2>s_{1}-n\geq0
\]
for $n>1071$.
\end{proof}

\begin{figure}[th]
\centering
\subfloat[$n>6$] {\begin{pspicture}(-1,-2)(3,2)\SpecialCoor
\psline[linestyle=dashed]{->}(0,-2)(0,2)
\psline[linestyle=dashed]{->}(-.5,0)(3,0)
\psline[linestyle=dashed](.5,1)(3,1)
\psline[linestyle=dashed](.5,-1)(3,-1)
\pscurve(3,.9)(1.7,.8)(1.2,.5)(1.7,.2)(3,.1)
\pscurve(3,-.9)(2,-.5)(3,-.1)
\rput[l](3,0){\small $\mu$}
\rput[r](0,2){\small $\nu$}
\rput[rt](1,1.5){\tiny $(+)$}
\rput[l](2.5,.5){\tiny $(-)$}
\rput[l](2.5,-.5){\tiny $(-)$}
\NormalCoor
\end{pspicture}}
\subfloat[$n=3,4,5,6$] {\begin{pspicture}(-1,-2)(3,2)\SpecialCoor
\psline[linestyle=dashed]{->}(0,-2)(0,2)
\psline[linestyle=dashed]{->}(-.5,0)(3,0)
\psline[linestyle=dashed](.5,1)(3,1)
\psline[linestyle=dashed](.5,-1)(3,-1)
\pscurve(3,.9)(2.5,.8)(1.3,0)(1,-.5)(1.3,-1)(2,-1.3)(2.5,-1.2)(3,-1.1)
\rput[l](3,0){\small $\mu$}
\rput[r](0,2){\small $\nu$}
\rput[rt](1,1.5){\tiny $(+)$}
\rput[l](2.5,.5){\tiny $(-)$}
\NormalCoor\end{pspicture}}\caption{Graph of $d_{1}(\mu,\nu)=0$.}%
\end{figure}
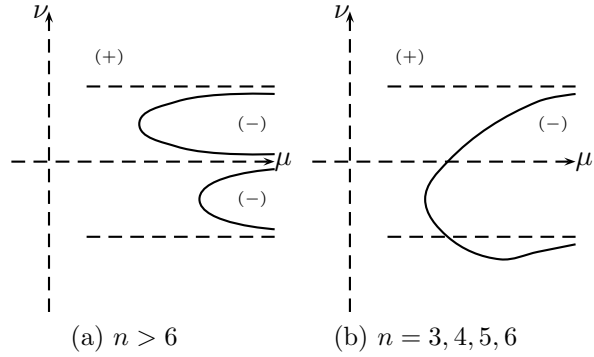

We shall analyze the graph of $d_{1}(\nu)=0$ only for positive masses, $\mu>0
$.

\begin{proposition}
For $n\geq7$, the polynomial $d_{1}(\mu,\nu)$ is zero only for the function%
\[
\mu_{0}(\nu):(-1,0)\cup(0,1)\rightarrow\mathbb{R}^{+}\text{,}%
\]
where $\mu_{0}(\nu)\rightarrow\infty$ when $\left\vert \nu\right\vert
\rightarrow0,1$ .

For $n\in\{3,4,5,6\}$, the polynomial $d_{1}(\mu,\nu)$ is zero only on the
curve
\[
(\mu_{0},\nu_{0}):\mathbb{R\rightarrow R}^{+}\times(-\infty,1)\text{,}%
\]
where $\mu_{0}(t)\rightarrow\infty$ and $\nu_{0}(t)\rightarrow\pm1$ when
$t\rightarrow\pm\infty$.
\end{proposition}

\begin{proof}
Since $\omega=\mu+s_{1}$,%
\[
d_{1}(\mu,\nu)=\nu^{2}\left(  \nu^{2}-1\right)  \mu^{2}+[s_{1}b_{1}%
-a_{1}+c_{1}+s_{1}^{2}\left(  \nu^{4}-\nu^{2}\right)  ]\mu/s_{1}+c_{1}\text{,}%
\]
where $s_{1}b_{1}-a_{1}=\left(  8s_{1}+s_{2}\right)  \left(  s_{1}-n\right)
/4$ and
\[
c_{1}(\nu)=[2s_{1}(\nu-1)^{2}+n][2s_{1}(\nu+1)^{2}+s_{2}]/4\text{.}%
\]

For each fixed $\nu$, set $d_{1}(\nu)$ be the value of the determinant. Then,
its derivative $d_{1}^{\prime}(\mu)$ is zero at the only critical point
$\mu_{c}$ such that%
\[
\frac{1}{s_{1}}(s_{1}b_{1}-a_{1}+c_{1}+s_{1}^{2}\left(  \nu^{4}-\nu
^{2}\right)  )\mu_{c}=-2\nu^{2}\left(  \nu^{2}-1\right)  \mu_{c}^{2}\text{.}%
\]
For $\left\vert \nu\right\vert <1$, since $d_{1}(\mu_{c},\nu)=-\nu^{2}\left(
\nu^{2}-1\right)  \mu_{c}^{2}+c_{1}$ is positive, then the polynomial
$d_{1}(\mu)$ has at most one zero for $\mu>\mu_{c}$. Since $d_{1}/\omega
^{2}\rightarrow\nu^{2}(\nu^{2}-1)$ when $\mu\rightarrow\infty$, and
$d_{1}(0,\nu)=c_{1}$ is positive, then the polynomial $d_{1}$ has exactly one
zero at $\mu_{0}(\nu)>0$ for each $\left\vert \nu\right\vert \in(0,1)$.

In order to obtain the limits of the function $\mu_{0}(\nu)$, when $\left\vert
\nu\right\vert \rightarrow0,1$, we need to study the zeros of $d_{1}(\mu,\nu)$
for $\left\vert \nu\right\vert =0,1$. For $\nu=0$, the polynomial $d_{1}%
(\mu,0)$ is zero at $\mu_{0}=-a_{1}/b_{1}$. Since $a_{1}$ is positive and
$b_{1}$ is negative for $n\leq6$ only, then $\mu_{0}$ is positive for $n\leq6$
only. For $\nu=1$, the polynomial $d_{1}(\mu,1)$ is zero at the point
$\mu_{+1}=-n$. For $\nu=-1$, the polynomial $d_{1}(\mu,-1)$ is zero at
\[
\mu_{-1}=-s_{2}(8s_{1}+n)/(8s_{1}-8n+9s_{2})\text{,}%
\]
which is positive for $n\leq6$ only.

Hence, for $n\geq7$, the polynomial $d_{1}(\mu,\nu)$ is positive at
$\left\vert \nu\right\vert =0,1$, for any $\mu\geq0$. Therefore, the functions
$\mu_{0}(\nu)\rightarrow\infty$ when $\left\vert \nu\right\vert \rightarrow
0,1$. Moreover, since $c_{1}(\nu)$ is increasing for $\nu>1$, then $d_{1}%
(\nu)\geq d_{1}(1)>0$, for $\nu>1$. Similarly, one has that $d_{1}(\nu)>0$,
for $\nu<-1$. We conclude that the only zeros of $d_{1}(\mu,\nu)$ for $\mu
\geq0$ are for the function $\mu_{0}(\nu)$.

For $n\leq6$, we may prove, in a similar way, that $d_{1}(\mu,\nu)$ is
positive, for $\nu>1$. As a consequence, the function $\mu_{0}(\nu
)\rightarrow\infty$ when $\nu\rightarrow1$. Also, the function $\mu_{0}(\nu)$
has a continuous extension, to $\nu=0,-1$, by $\mu_{0}(0)=\mu_{0}$ and
$\mu_{0}(-1)=\mu_{-1}$. Furthermore, the continuum of zeros of $d_{1}(\mu
,\nu)$ at $(\mu_{-1},-1)$ crosses the line $(\mu,-1)$ only once, then this
continuum must go to $(\mu,\nu)=(\infty,-1)$. This statement follows from the
fact that $d_{1}(\mu,\nu)\ $ is positive for $\nu<-1$ and $\mu=0$, for $\mu>0$
and $\left\vert \nu\right\vert $ large, and for $\nu<-1$ and $\mu$ large.
\end{proof}

\begin{remark}
The equation $d_{1}(\mu,\nu)=0$ is quadratic in $\mu$, thus we can find
explicitly two roots $\mu_{\pm}(\nu)$of $d_{1}(\mu,\nu)=0$. In principle, we
could analyze these roots analogously to the case $k\in\{2,...,n-2\}$.
However, there are many cases depending on $n$. The main case is for
$n\geq473$, where we may prove that the graph of $d_{1}(\mu,\nu)=0$ looks like
the graph of $d_{k}(\mu,\nu)=0$, for $\mu>-s_{1}$.
\end{remark}

\begin{definition}
For $n\geq7$, define $m_{+}$ to be the minimum of the function $\mu_{0}(\nu)$
in $(0,1)$, and $m_{-}$ the minimum in $(-1,0)$. For $n\in\{3,...,6\}$ define
$m_{0}$ to be the minimum of $\mu_{0}(t)$ for $t\in\mathbb{R}$.
\end{definition}

\begin{proposition}
For masses $\mu>0$ we have the following cases:

\begin{description}
\item[(a)] For $n\in\{3,...,6\}$, the matrix $m_{1}(\nu)$ changes its Morse
index at the two values $\nu_{-}<\bar{\nu}_{-}<0$ if $\mu\in(m_{0},\mu_{1})$,
and at $\nu_{-}<0<\nu_{+}<1$ if $\mu\in(\mu_{1},\infty)$.

\item[(b)] For $n\geq12$, the matrix $m_{1}(\nu)$ changes its Morse index at
$0<\bar{\nu}_{+}<\nu_{+}<1$ if $\mu\in(m_{+},m_{-})$, and at the four values
of $\nu$
\[
-1<\nu_{-}<\bar{\nu}_{-}<0<\bar{\nu}_{+}<\nu_{+}<1
\]
if $\mu\in(m_{-},\infty)$. For $n\in\{7,...,11\}$, this statement is true with
$m_{+}$ and $m_{-}$ interchanged.
\end{description}

In both cases, we have that%
\[
\eta_{1}(\bar{\nu}_{\pm})=-1\text{ and }\eta_{1}(\nu_{\pm})=1.
\]

\end{proposition}

\begin{proof}
For $n\geq12$, using $s_{2}>n$, we may prove that $c_{1}(\nu)<c_{1}(-\nu)$.
Then, $d_{1}(\mu,\nu)<d_{1}(\mu,-\nu)$, for $\nu>0$. Therefore, $m_{+}<m_{-}$,
for $n\geq12$. In a similar way, we may prove that $m_{-}>m_{+}$, for $7\leq
n\leq11$.

The trace of $m_{1}(\nu)$ is
\[
T_{1}=\mu\omega(\nu-1)^{2}+2\omega(\nu^{2}+1)+\mu(n/2+1)+2\alpha_{1}\text{.}%
\]
Since $T_{1}$ is positive for $\mu>0$, then the Morse index satisfies
$n_{1}(\nu)=1$, when $\det m_{1}(\nu)$ is negative, and $n_{1}\in\{0,2\}$,
when $\det m_{1}(\nu)$ is positive.

Let $\Omega$ be the connected component where $d_{1}(\mu,\nu)$ is negative,
for $n\geq7$ this set is $\Omega=\{(\mu,\nu):\mu>\mu_{0}(\nu)\}$. From the
definition, we have that $n_{1}=1$ in $\Omega-\{(\mu,\pm1)\}$. Now, for
$\mu=0$, the matrix $m_{1}(0,\nu)$ has eigenvalues $0$, $s_{1}\left(
\nu-1\right)  ^{2}$ and $s_{1}\left(  \nu+1\right)  ^{2}+2\alpha_{1}$, hence
$n_{1}(\varepsilon,\nu)\leq1$, for small $\varepsilon$. Since $n_{1}\in\{0,2\}
$ and $n_{1}\leq1$ in $\bar{\Omega}^{c}-\{(\mu,\pm1)\}$, then $n_{1}=0$ there.
Thus, we have that $\eta_{1}(\pm1)=0-0$, $\eta_{1}(\bar{\nu}_{\pm})=0-1$ and
$\eta_{1}(\nu_{\pm})=1-0.$
\end{proof}

From the bifurcation theorem, we get the following result:

\begin{theorem}
For $n\geq12$,\ if $\mu>m_{+}$, the polygonal equilibrium has two global
bifurcation branches of planar periodic solutions with symmetries
$\mathbb{Z}_{n}\left(  \zeta,\zeta,-\zeta\right)  $. If $\mu>m_{-}$, there are
two more bifurcation branches with symmetries $\mathbb{Z}_{n}\left(
\zeta,\zeta,\zeta\right)  $. For $n\in\{7,...,11\}$, the same statement is
true interchanging $m_{+}$ and $m_{-}$.

For $n\in\{3,...,6\}$, if $\mu>\mu_{1}$, the polygonal equilibrium has one
global bifurcation branch of planar periodic solutions with symmetries
$\mathbb{Z}_{n}\left(  \zeta,\zeta,-\zeta\right)  $, and one more with
symmetries $\mathbb{Z}_{n}\left(  \zeta,\zeta,\zeta\right)  $. If $\mu
\in(m_{0},\mu_{1})$,\ there are two global bifurcation branches of planar
periodic solutions with symmetries $\mathbb{Z}_{n}\left(  \zeta,\zeta
,\zeta\right)  $.
\end{theorem}

\begin{remark}
When $\mu=0$, the central body does not exist, then the block $m_{1}(\nu)$ is
\[
m_{1}(\nu)=\left(
\begin{array}
[c]{cc}%
s_{1}\nu^{2}+s_{1}+\alpha_{1} & \alpha_{1}i+2\nu s_{1}i\\
-\alpha_{1}i-2\nu s_{1}i & s_{1}\nu^{2}+s_{1}+\alpha_{1}%
\end{array}
\right)  \text{.}%
\]
In this case $m_{1}(\nu)$ has eigenvalues $s_{1}(\nu-1)^{2}$ and $s_{1}%
(\nu+1)^{2}+2\alpha_{1}$. Then $n_{1}(\nu)=0$ for all $\nu$, and there is no
bifurcation in this case.
\end{remark}

\section{Comments}

\subsection{Case $n=2$}

The irreducible representations for $n=2$ are different from those for larger
$n$, due to the action of $\mathbb{Z}_{2}$. This fact gives a change in the
planar representation for $k=1$, but the change of variables for $k=2$ remains
the same as for larger $n$'s. In particular, the bifurcation results for the
spatial case and for the planar case, with $k=2$ are similar to the ones
already given.

For $k=1$, define the isomorphism $T_{1}:\mathbb{C}^{4}\rightarrow W_{1}$ as%
\begin{align*}
T_{1}(v,w)  &  =(v, 2^{-1/2}w, 2^{-1/2}w)\text{ with}\\
W_{1}  &  =\{(v,w,w):v,w\in\mathbb{C}^{2}%
\end{align*}

It is not difficult to find the action of $\mathbb{Z}_{2}$ on $W_{1}$ and to
compute the Hessian of the potential $V$. The matrix $B_{2}$ is the same as
before, but $B_{1}$ is now a $4\times4$ complex matrix. For $\alpha=2$, one
has that $s_{1}=1/4$ and $\omega=\mu+1/4$. The matrix $m_{1}(\nu)$ is%

\[
\left(
\begin{array}
[c]{cccc}%
\mu\left(  \omega\nu^{2}+\mu+17/4\right)  & -2i\nu\omega\mu & -2\left(
2\right)  ^{1/2}\mu & 0\\
2i\nu\omega\mu & \mu\left(  \omega\nu^{2} +\mu-7/4\right)  & 0 & \left(
2\right)  1/2\mu\\
-2\left(  n/2\right)  ^{1/2}\mu & 0 & \omega\nu^{2}+3\mu+ 1/4 & -2i\nu\omega\\
0 & \left(  2\right)  ^{1/2}\mu & 2i\nu\omega & \omega\nu^{2}+ 1/4
\end{array}
\right)  \text{.}%
\]

The determinant of $m_{1}(\nu)$ may be written as%
\[
\det m_{1}(\nu)= 2^{-8}\mu^{2}(4\mu+1)^{2}(\nu^{2}-1)^{2}d_{1}(\nu),
\]
where $d_{1}(\nu)$ is the polynomial
\[
d_{1}(\nu)=\left(  16 \mu^{2} + 8 \mu+1\right)  \nu^{4} + \left(  -16 \mu^{2}
+20 \mu+6\right)  \nu^{2} -\left(  84\mu+119\right)  \text{.}%
\]

For $\mu>0$, $d_{1}(\nu)$ is zero only for $\pm\nu_{1}$, where $\nu_{1}$ is
the positive solution of%
\[
\nu_{1}^{2} = \left(  2\mu-3 + 2\left(  \mu^{2}+18\mu+32\right)
^{1/2}\right)  /(4\mu+1)\text{.}%
\]

Since $\nu_{1}>1$ for $\mu>0$, it is enough to know the eigenvalues of
$m_{1}(\nu)$ for one value of $\mu$, for instance, $\mu=1$. These eigenvalues
are $5(\nu\pm1)^{2}$ and $\left(  5\nu^{2} \pm2(25\nu^{2}+81)^{1/2}
+11\right)  /4$.

Thus, one sees that $n_{1}=1$ for $-\nu_{1}<\nu<\nu_{1}$ and $n_{1}=0$
otherwise. Hence, $\eta_{1}(\nu_{1})=1$ and $\eta_{1}(1)=0$. Therefore,

\begin{theorem}
There is a global bifurcation branch of planar periodic solutions, from
$2\pi/\nu_{1}$, with symmetries $\mathbb{Z}_{2}\left(  \zeta,\pi,\pi\right)
$. There is also a global bifurcation of planar periodic solutions, from the
period $2\pi$, and symmetry $\mathbb{Z}_{2}\left(  \zeta,\pi,0\right)  $, that
is
\[
u_{0}(t)=0\text{ and }u_{2}(t)=-u_{1}(t).
\]
Finally, there are two spatial bifurcation branches, from $\nu=\sqrt{\mu+1/4}$
and $\nu=\sqrt{\mu+2}$, the first type is such that $u_{0}(t)=0$ and
$z_{0}(t)=0$,
\[
u_{2}(t)=-u_{1}(t)\text{ and }z_{2}(t)=-z_{1}(t),
\]
while, the second type is with $u_{0}(t)=0$ and $z_{0}(t+\pi)=z_{0}(t)$,
\[
u_{2}(t)=-u_{1}(t)=-u_{1}(t+\pi)\text{ and }z_{2}(t)=z_{1}(t)=-z_{1}(t+\pi).
\]

\end{theorem}

\subsection{Resonances}

We have proved that there is a bifurcation of periodic solutions with the
symmetries (\ref{D3.2.0}) when the block $M_{1}(\nu)$ changes its Morse index
at $\nu_{1}$. However, we cannot deduce, from the symmetries (\ref{D3.2.0}),
that the solutions are truly spatial, because all the spatial coordinates can
be zero: $z_{j}(t)=0$. In fact, there is a non-resonance condition to assure
that the solutions are truly spatial.

\begin{proposition}
If $M_{1}(\nu)$ changes its Morse index at $\nu_{1}$, and the matrix
$M_{0}(2l\nu_{1})$ is invertible for all $l\in\mathbb{N}$, then the
equilibrium $x_{0}$ has a bifurcation of truly spatial periodic solutions from
$\nu_{1}$ with symmetries (\ref{D3.2.0}) .
\end{proposition}

\begin{proof}
We need to prove that near $(\bar{a},\nu_{k})$ there are no planar solutions
with symmetry $\mathbb{\tilde{Z}}_{2}$, i.e.%
\[
x_{j}(t)=x_{j}(t+\pi)\text{, }y_{j}(t)=y_{j}(t+\pi)\text{ and }z_{j}%
(t)=0\text{.}%
\]
These solutions have isotropy group $H=\mathbb{Z}_{2}(\kappa)\times
\mathbb{Z}_{2}(\pi)$, one generated by $\kappa$ and the other by $\pi\in
S^{1}$.

Let $f^{\prime}(\bar{a})$ be the derivative of the bifurcation operator
$f(x)$, at $\bar{a}$, and define $f^{\prime}(\bar{a})|_{\mathcal{W}^{H}}$ as
the derivative in the fixed point space of $H$. The linear map $f^{\prime
}(\bar{a})|_{\mathcal{W}^{H}}$ has a block of the form $M_{0}(2l\nu)$, for
each Fourier mode $l$.

By assumption, the blocks $M_{0}(2l\nu_{1})$ are invertible, then the
derivative $f^{\prime}(\bar{a})|_{\mathcal{W}^{H}}$ is invertible. Using the
implicit function theorem, we may conclude the non-existence of periodic
solutions with isotropy group $H$ near $(\bar{a},\nu_{1})$. As a consequence,
the bifurcation from $\nu_{1}$ must be spatial.
\end{proof}

For the polygonal equilibrium, we have proved before that the block
$m_{1k}(\nu_{k})$ changes its Morse index at $\nu_{k}=\sqrt{\mu+s_{k}}$.
Restricting $f^{\prime}(\bar{a})$ to the group $\mathbb{\tilde{Z}}%
_{n}(k)\times\mathbb{Z}_{2}\times\mathbb{Z}_{2}$, we may prove, analogously to
the general case, that the bifurcation from $\nu_{k}$ is truly spatial if
$m_{0k}(2l\nu_{k})$ is invertible for all $l^{\prime}$s.

For $k\in\{1,...,n-1\}$, we have proved that the block $m_{0k}(\nu)$ is
invertible for $\left\vert \nu\right\vert >\sqrt{\mu+s_{1}}$, when $\mu>m_{0}%
$: $\nu=1$ in the previous graph corresponds to $\nu=\sqrt{\mu+s_{1}}$ before
the normalization used in the study of $m_{k}(\nu)$. Since $2l\nu_{k}%
>\sqrt{\mu+s_{1}}$ for all $l$'s, then $m_{0k}(2l\nu_{k})$ is invertible for
all $l$'s, when $\mu>m_{0}$. Therefore, the bifurcation from $\nu_{k}$ is
truly spatial for $k\in\{1,...,n-1\}$, when $\mu>m_{0}$.

If one uses the same normalization done for the planar spectrum, that is
$m_{0k}(\sqrt{\omega} \nu)$, for the spatial spectrum, then one has that, for
$k=1,...,n-1$,
\[
m_{1k}(\sqrt{\omega} \nu)= {\nu}^{2} (\mu+s_{1})-(\mu+s_{k})\text{,}%
\]
and the corresponding expression for $k=n$.

Then, the block $m_{1k}$ changes its Morse index at $\nu_{k}= \sqrt{(\mu
+s_{k})/(\mu+s_{1})}$. Restricting $f^{\prime}(\bar{a})$ to the group
$\mathbb{\tilde{Z}} _{n}(k)\times\mathbb{Z}_{2}\times\mathbb{Z}_{2}$, we may
prove, analogously to the general case, that the bifurcation from $\nu_{k}$ is
truly spatial if $m_{0k}(2l\nu_{k})$ is invertible for all $l$'s.

For $k\in\{1,...,n-1\}$, we have proved that the block $m_{0k}(\nu)$ is
invertible for $\left\vert \nu\right\vert >1$, when $\mu>m_{0}$. Since
$2l\nu_{k}>1$ for all $l$'s, then $m_{0k}(2l\nu_{k})$ is invertible for all $l
$'s, when $\mu>m_{0}$. Therefore, the bifurcation from $\nu_{k}$ is truly
spatial for $k\in\{1,...,n-1\}$, when $\mu>m_{0}$. For other values of $\mu<
m_{0}$, one has that ${\nu}_{k}^{2}=(\mu+s_{k})/(\mu+s_{1})$, and $\omega
=\mu+s_{1}=(s_{k}-s_{1})/({\nu_{k}}^{2}-1)$. The equation $d_{k}(2l\nu_{k})=0$
leads to finding the zeros of a $4$th degree polynomial in $\nu_{k}$, with
coefficients which depend on $l$. For the piece of the graph where $\nu_{0}%
<1$, one cannot have the resonance $\nu_{0}= 2l\nu_{k}$ and, for the rest of
the graph which is asymptotic to the line $\mu=-s_{1}$, there will be only a
finite number of $l$'s, where one could have the equality of the frequencies.
As a consequence, there is only a finite number of possible resonances, for
$k=1,..., n-1$.

For $k=n$, ${\nu_{n}}^{2}=(\mu+n)/(\mu+s_{1})$ and $\nu_{0}=1$, so that there
is at most a finite number of $l$'s and positive $\mu$'s which satisfy the
relation $\nu_{0}=2l\nu_{n}$.

\begin{remark}
If one has a resonance between spatial frequencies, for $k_{1}$ and $k_{2}$,
such that $\nu_{2}=l\nu_{1}$, then one has the relation $\mu(l^{2}%
-1)=s_{k_{2}}-l^{2}s_{k_{1}}$, where $s_{k_{2}}$ is replaced by $n$ for
$k_{2}=n$. If $l=1$ and $k_{2}\neq n$, then $k_{2}=n-k_{1}$, and the
bifurcating orbit for $k_{2}$ is the one for $k_{1}$ taken in the opposite
direction, due to the symmetry $\tilde{\kappa}$.

If $l=1$ and $k_{2}=n$, then one needs $s_{k_{1}}=n$ for a $1:1$ resonance. We
know that $s_{n/2}>...>s_{1}$, and we have seen, in the last lemma
\label{5EnLe}, that $s_{1}>n\ $ for $n>472$, that $s_{2}>n>s_{1}$ for $12\leq
n\leq472$ and $n>s_{2}$ for $3\leq n\leq11$. Furthermore, numerically one
computes $n>s_{3}$ for $6\leq n\leq7$ and $s_{3}>n$ for $8\leq n\leq11$. Since
all these inequalities between $s_{0}=n$ and $s_{k}$ are strict, then we may
conclude that there are strict inequalities between $\nu_{0}$ and $\nu_{k_{1}%
}$, that is there is no $1:1$ resonance in this case.

For $k_{2}<n$, we may assume that $1\leq k_{1}<k_{2}\leq n/2$ and we need, for
$\mu\geq0$, to have $l^{2}\leq s_{k_{2}}/s_{k_{1}}$. Using the recurrence
formulae of the appendix, we see that $s_{k_{2}}/s_{k_{1}} < k_{2}^{2}%
-k_{1}^{2}+1\leq k_{2}^{2}$. Hence, the number of possible subharmonic
resonances is finite, as well as the number of possible solutions to the
relation $\mu= (s_{k_{2}}-l^{2}s_{k_{1}})/(l^{2}-1)$.
\end{remark}

\begin{remark}
For resonances between planar frequencies, if $\nu_{2}=l\nu_{1}$ and
$d_{k_{j}}(\nu_{j})=0$ for $j=1,2$, it is easy to see that $\nu$ on the graph
of $d_{k}(\nu)=0$ is bounded above, for $\mu\geq0$ and away from $0$ for all
positive bounded $\mu$'s, except for small neighborhoods of $\mu_{k}$. Thus,
the number of possible $l$'s is finite. Now, if $k_{j}<n$, $d_{k_{2}}(l\nu
_{1})-l^{2}d_{k_{1}}(\nu_{1})$ is a polynomial of degree $2$ in $\mu$ and
$\nu_{1}$, which can be solved for $\mu$ in terms of $\nu_{1}$ and $d_{k_{1}%
}(\nu_{1})$ gives a polynomial of degree $16$ in $\nu_{1}$, if $l>1$, and of
degree $4$ if $l=1$. In any case, we obtain at most $16$ solutions for
$\nu_{1}$ and $32$ for $\mu$, i.e., a finite number of possible resonances,
with $l$ bounded if $\mu$ is positive, bounded, and not in a neighborhood of
$\mu_{k_{1}}$, where $\nu_{1}$ is small.

Note that, if $k_{2}=n$, then $\nu_{2}=1$ and there is a resonance for
$\nu_{1}=1/l$ and the corresponding $\mu$, with $l$ going to infinity as $\mu$
goes to $\mu_{k_{1}}$. For $k_{1}\leq n/2$, assume that the graph for
$d_{k_{2}}(\nu)$ covers a small interval $[\mu_{k_{1}}, \tilde\mu]$. Then, if
$\tilde{l}$ is such that $\tilde{l} \nu_{k_{1}}(\tilde{\mu})\geq nu_{k_{2}%
}(\tilde{\mu})$, the graphs for $d_{k_{1}}(l\nu)$ will cut the graph for
$d_{k_{2}}(\nu)$, for any $l\geq\tilde{l}$, in the interval, at $\mu(l)$ with
the limit $\mu_{k_{1}}$, giving rise to an infinite number of resonances. This
situation will happen if $\mu_{k_{2}}>\mu_{k_{1}}$, which is the case for $n$
large enough (see remark 29 in \cite{GaIz11}). But, for $\mu$ different from
$\mu_{k_{1}}$, one has a finite number of possible resonances.

A similar argument may be done for $k_{1}>n/2$ and intervals at the left of
$\mu_{k_{1}}$.

Note that, at $\mu_{k_{1}}$ one has a two-dimensional kernel for the
linearization at the relative equilibrium (for $k_{1}$ and $n$) and the orbit
is no longer hyperbolic, so one needs a different computation of the
bifurcation index.

Assume that there is no resonance, that is that $d_{j}(l\nu_{k})\neq0$, for
$j\neq k$ and any $l\geq1$, and $\nu_{k}\neq1$. Then the equations for $j\neq
k$, $j\neq n$ may be solved close to the bifurcation orbit, by an implicit
function argument in terms of the component for $k$. The equation for $n$,
which has a one dimensional kernel, is solved by a Poincar\'{e} section
argument which is used in \cite{IzVi03} in order to compute the orthogonal
index for an isolated hyperbolic orbit (theorem 3.1 p. 247 in the above
reference). One has then a reduction to one mode, a single $k$ and a
one-dimensional complex orthogonal equation. The non-trivial solutions of this
reduced system have a least period equal to $2\pi$, which implies that, except
for a finite number of $\mu$'s, if different from $\mu_{k}$, the branches
which bifurcate from the relative equilibrium, for fixed $\mu$, will be different.
\end{remark}

\subsection{Linear stability}

The linearization of the equation at the equilibrium $x_{0}$, is%
\begin{equation}
\mathcal{M}\ddot{u}+2\sqrt{\omega}\mathcal{M\bar{J}}\dot{u}=D^{2}V(x_{0})u.
\label{D7.2.0}%
\end{equation}
For this problem, an equilibrium is linearly stable if the linear equation
(\ref{D7.2.0}) has $6(n+1)$ linearly independent solutions of the form
$u(t)=e^{i\nu_{0}t}v_{0}$ with $\nu_{0}\in\mathbb{R}$, for $n+1$ bodies.

Now, if $\nu_{0}\in\mathbb{R}$ is a bifurcation value, that is the determinant
is $0$ at $\nu_{0}$, then there is a vector $v_{0}$ such that $M(\nu_{0}%
)v_{0}=0$. Since%
\[
-\mathcal{M}\ddot{u}-2\sqrt{\omega}\mathcal{M\bar{J}}\dot{u}+D^{2}%
V(x_{0})u=e^{i\nu_{0}t}M(\nu_{0})v_{0}=0\text{,}%
\]
then the function $u(t)=e^{i\nu_{0}t}v_{0}$ is a solution of the linearized
equation (\ref{D7.2.0}).

Therefore, for each point $\nu_{0}$ where $M(\nu)$ is singular, the
equilibrium has a direction where the linearized equation is stable. Now,
$\nu_{0}$ may be a multiple zero of the characteristic determinant and one
could have a polynomial increase in $t$. A weaker definition of stability is
that of spectral stability, where the equilibrium $x_{0}$ is spectrally stable
if $M(\nu)$ is non-invertible at $6(n+1)$ values of $\nu$, counted with multiplicities.

As a consequence, we may deduce the spectral stability from the spectral
analysis which we did to prove bifurcation. We shall recover, with somewhat
different arguments, the results of \cite{Ro00}, of \cite{VaKo07} and others.

For the polygonal equilibrium, we may give a full description of stability. We
have proved that each spatial block $m_{1k}(\nu)$ is non-invertible at exactly
one positive value, for $k\in\{1,...,n\}$, and $m_{1n}(0)$ has a
one-dimensional kernel, with a double zero. Then, the spatial spectrum gives
$n$ stable solutions. The kernel of $m_{1n}(0)$ is generated by the vector
$(1,\sqrt{n})$, which gives, under the transformation $P$ the vector
$(1,1,,,,1)=e$. The eigenvectors corresponding to the double root are
$(a+bt)e$. In fact, since $\sum_{i}a_{ij}=0$, the vector $e$ generates the
kernel of the vertical Hessian.

Actually, if one sums all the vertical components of the equations, one gets
$\mu\ddot{z}_{0}+\sum_{j=1}^{n}\ddot{z_{j}}=0$ and, if the initial conditions
for $\mu z_{0}+\sum_{j=1}^{n}z{_{j}}$ are $0$, for position and speed, then
this sum remains $0$. Hence, by imposing these restrictions, one obtains a
linear stability in the vertical direction.

Next, we analyze the planar spectrum.

\begin{proposition}
Let $m_{\ast}$ be the largest of the $m_{+}$'s, for $k=1,...,n$. Then, the
polygonal equilibrium is spectrally stable only if $n\geq7$ and $\mu>m_{\ast}
$. It will be linearly stable if one restricts vertical translations,
rotations around the center of mass and scaling of the frequency.
\end{proposition}

\begin{proof}
For $k\in\{2,...,n-2\}$, we have proved that the $2\times2$ matrix $m_{k}%
(\nu)$ is non-invertible at four values of $\nu$, when $\mu>m_{\ast}$. For
$k=n$, the $2\times2$ matrix $m_{n}(\nu)$ is non-invertible at one positive
and one negative value, and $0$ is a double root, with a kernel generated by
the vector $(0,1)$. Applying the transformation $P$, one gets the eigenvectors
$Jx_{0}$ and $2x_{0}-3\sqrt{\omega}tJx_{0}$. Now, consider the relation
$\nabla V(\omega a^{-3},ax_{0})=0$, for any positive $a$, and differentiate
it, with respect to $a$, in order to obtain that $D^{2}V(x_{0})x_{0}%
=3\mathcal{M}\omega x_{0}$. Since $Jx_{0}$ generates the kernel of
$D^{2}V(x_{0})$, one verifies that the two vectors are solutions of the
linearized equation. Hence, the behavior at $\nu=0$ is a consequence of the
symmetry of rotations in the $(x,y)$-plane and scaling. Thus, this does not
destabilize the polygonal equilibrium, if one asks for this natural restriction.

For $k\in\{1,n-1\}$, we proved that the $3\times3$ matrix $m_{k}(\nu)$ is
non-invertible at three values, with $\nu=1$ as a double root, for
$n\in\{3,...,6\}$. Hence, since the total count of zeros is $4n$, the
polygonal equilibrium is never spectrally stable for $n\in\{3,...,6\}$.
Similarly for $n=2$, the $4\times4$ matrix $m_{1}(\nu)$ has a characteristic
determinant with six real and two purely imaginary zeros, giving an
exponential instability.

For $n\geq7$, we proved that $m_{1}(\nu)$ is non-invertible at six values of
$\nu$, with $\nu=1$ being a double root, if $\mu>m_{\ast}$. Therefore, the
polygonal equilibrium is spectrally stable only if $\mu>m_{\ast}$, since one
has the complete count of $4(n+1)$ roots for the planar spectrum.

Now, it is easy to prove that, for $\nu=1$, the matrix $m_{1}(1)$ has a kernel
generated by the vector $(\sqrt{2/n},1,i)$, which gives, under the
transformation $P$, the vector $v=1/\sqrt{n}((1,i),(1,i),...,(1,i))$. It is
then easy to prove that $(a+bt)e^{\sqrt{\omega}t}v$ are solutions of the
linearized equations. In fact, if one takes the sum of all the $2\times2$
matrices $\sum_{j}A_{ij}=0$, one has that $v$ is in the kernel of the Hessian
of the Newton's potential, and $(a+bt)e^{\sqrt{\omega}t}(1,i)$ generates the
four (in its real and imaginary parts) solutions of the equation
\[
\ddot{x}+2\sqrt{\omega}J\dot{x}-\omega x=0.
\]
Furthermore, since Newton's potential is invariant under rotations in the
plane of all the masses, then its gradient is equivariant and its Hessian has
$v$ as generator of its kernel. This explains the restriction on the rotations.
\end{proof}

Assume the assertion, in the paper \cite{Ro00}, that the largest of the
$m_{+}$'s corresponds to the block $m_{n/2}(\nu)$.

In the Theorem \ref{E6.3.1}, if we had proved that there are only three
solutions for $d_{k}(\nu)=0$ and $d_{k}^{\prime}(\nu)=0$, then these solutions
would have been $m_{0}$ and $m_{\pm}$. In that case, the only points where the
Morse index changes are those of the Theorem \ref{E6.3.1}. We can prove this
result for $k=n/2$, and large $n$, since $\gamma_{n/2}=0$. In fact, one has
that $m_{+}=m_{-}$, and $m_{0}=\mu_{n/2}$ if $n$ is large enough. Since
$\gamma_{n/2}=0$, the graph for $d_{n/2}$ is symmetric with respect to the
$\mu$-axis and one has two curves for $\mu$ as a function of $\nu$.

In order to compute $m_{+}=m_{-}$, notice that%
\[
d_{k}^{\prime}(\nu)=2\nu\omega^{2}(2\nu^{2}+\omega^{-1}(2\alpha_{n/2}%
-2s_{1}-\mu))=0\text{.}%
\]
Substituting $\nu^{2}=-\omega^{-1}(2\alpha_{n/2}-2s_{1}-\mu)/2$ in $d_{k}%
(\nu)=0$, we obtain%
\[
(2\alpha_{n/2}-2s_{1}-\mu)^{2}-4(a_{n/2}+\mu b_{n/2})=0\text{.}%
\]
From this quadratic equation, we get that $m_{+}=b+\sqrt{b^{2}+c}$, with
\[
b=2(b_{n/2}+\alpha_{n/2}-s_{1})\text{ and }c=4(a_{n/2}-(\alpha_{n/2}%
-s_{1})^{2}).
\]
Next, we wish to find the estimate of $m_{+}$, given in \cite{Ro00}, on the
stability of Saturn's rings.

\begin{lemma}
\cite{Ro00}: the sums $s_{k}$ satisfy the limits
\[
\lim_{n\rightarrow\infty}\frac{s_{1}}{n\ln n}=\frac{1}{2\pi}\text{ and }%
\lim_{n\rightarrow\infty}\frac{s_{n/2}}{n^{3}}=\sigma=\frac{1}{2\pi^{3}}%
\sum_{k=1}^{\infty}\frac{1}{(2k-1)^{3}}\text{. }%
\]

\end{lemma}

From proposition (\ref{EA.1.0}), given in the appendix, where $\bar{s}_{l}$ is
defined and $\alpha$ is taken to be $2$, we have
\[
s_{n/2}-s_{n/2-1}=(n-1)s_{1}-\sum_{l=1}^{n/2-1}\bar{s}_{l}\text{.}%
\]
Using approximations by integrals, we can prove that $\bar{s}_{l}/n$ is finite
and that $s_{1}/n\rightarrow\infty$ when $n\rightarrow\infty$. Since $\bar
{s}_{l}=O(s_{1})$, when $n\rightarrow\infty$, we have that%
\[
s_{n/2}-s_{n/2-1}=O(ns_{1})=O(n^{2}\ln n)=o(n^{3})\text{.}%
\]
Consequently, we have that $\lim_{n\rightarrow\infty}s_{n/2-1}/n^{3}=\sigma$.

Using these limits and $\gamma_{n/2}=0$, we have that%
\[
\alpha_{n/2}/n^{3}\rightarrow\sigma/2\text{ and }\beta_{n/2}/n^{3}%
\rightarrow(3/2)\sigma\text{,}%
\]
when $n\rightarrow\infty$. From the definitions of $a_{n/2}$ and $b_{n/2}$, we
obtain that $a_{n/2}/n^{6}\rightarrow-2\sigma^{2}$ and $b_{n/2}/n^{3}%
\rightarrow6\sigma$. Therefore, $b/n^{3}\rightarrow13\sigma$ and
$c/n^{6}\rightarrow-9\sigma^{2}$, when $n\rightarrow\infty$. We conclude that%
\[
\lim_{n\rightarrow\infty}m_{+}/n^{3}=(13+4\sqrt{10})\sigma\text{.}%
\]
This limit is the one found in \cite{Ro00} in order to estimate the stability
of Saturn's rings.

\subsection{Non-abelain actions}

In rotating coordinates, the full group is $T^{2}\cup\bar{\kappa}T^{2}$ for a
general relative equilibrium, and the additional group of permutations $S_{n}$
for the polygonal relative equilibrium. The action of the element $\bar
{\kappa}\in\mathbb{Z}_{2}$ is given by
\[
\bar{\kappa}x_{j}(t)=diag(1,-1,1)x_{j}(-t).
\]
If one fixes the bifurcation map by the action of $\bar{\kappa}$, then the
equilibrium must be a collinear central configuration, which is not the case
of the polygonal equilibrium, or of the more general setting.

However, the polygonal equilibrium is fixed by the action
\[
\tilde{\kappa}x_{j}(t)=diag(1,-1,1)x_{n-j}(-t)\text{,}%
\]
which is a coupling of $\bar{\kappa}$ with the permutation of the bodies in
$D_{n}$. If one restricts the problem to the fixed-point subspace of
$\tilde{\kappa}$, then one may prove bifurcation of periodic solutions, but
only for the blocks $m_{0k}$ for $k=1,...,n$ and $m_{1k}$ for $k=n/2,n$.

Specifically, the bifurcation operator $f(x)$ is equivariant under the action
of the group
\[
(\mathbb{Z}_{n}\times T^{2})\times\tilde{\kappa}(\mathbb{Z}_{n}\times
T^{2})\text{,}%
\]
where the action of $\tilde{\kappa}$ is given, for the $n$ equal mass bodies,
modulus $n$, by
\[
\tilde{\kappa}x_{j}(t)=Rx_{n-j}(-t)\text{ with }R=diag(1,-1,1)\text{,}%
\]
while, for the central body, $j=0$, the action is
\[
\tilde{\kappa}x_{0}(t)=Rx_{0}(-t)\text{.}%
\]

Let $z_{k}$ be the coordinate for the block $m_{0k}$.The action of
$\tilde{\kappa}$ on $z_{k}$ is given by
\[
\tilde{\kappa}z_{k}=R\bar{z}_{k}\text{,}%
\]
where $R=diag(1,-1)$ for $k\notin\{1,n-1\}$ and $R=diag(1,1,-1)$ for
$k\in\{1,n-1\}$. Now, since the polygonal equilibrium $\bar{a}$ is fixed by
$\tilde{\kappa}$, then the linearization of the bifurcation map $f(x)$,
represented by the blocks $m_{0k}$, must be $\mathbb{Z}_{2}\mathbb{(}%
\tilde{\kappa})$-equivariant. This implies that $m_{0k}=R\bar{m}_{0k}R$.

Now, if we restrict the bifurcation operator $f(x)$ to the fixed-point space
of $\tilde{\kappa}$, then we get that the coordinates $z_{k}$ of the blocks
$m_{0k}$ satisfy
\[
z_{k}=R\bar{z}_{k}\text{.}%
\]
One may see that $M_{0}(\nu)$, in the fixed-point space of $\tilde{\kappa}$ is
equivalent to the matrix
\[
(m_{01},...,m_{0n})\text{,}%
\]
where $m_{0k}$ is defined in the space $\mathbb{R}\times i\mathbb{R}$.
Applying degree in the fixed-point space of $\tilde{\kappa}$, one proves
bifurcation of periodic solutions for the blocks $m_{0k}$. Moreover, since the
action of $\tilde{\kappa}$ and
\[
(\zeta,\zeta,-k\zeta)\in\mathbb{Z}_{n}\times S^{1}\times S^{1}%
\]
on these subspaces in $m_{0n}$ is trivial, then one can deduce that the
equilibria will have periodic solutions with the symmetries that we described
in addition to those of $\mathbb{Z}_{2}(\tilde{\kappa})$.

In the case of the spatial blocks, let $z_{k}$ be the the coordinate of the
block $m_{1k}$. The action on this coordinate is given by
\[
\tilde{\kappa}(z_{k},z_{n-k})=(\bar{z}_{n-k},\bar{z}_{k})\text{,}%
\]
for $k=1,...,n-1$ and $\tilde{\kappa}{z_{n}}=\bar{z}_{n}$. Since the blocks
$m_{1k}$ must be $\mathbb{Z}_{2}\mathbb{(}\tilde{\kappa})$-equivariant, then
one gets the equality
\[
m_{1k}=\bar{m}_{1(n-k)}\text{.}%
\]
Thus, when restricting the bifurcation operator $f(x)$ to the fixed-point
space of $\tilde{\kappa}$, the coordinates $z_{k}$ of the blocks $m_{1k}$ are
related by $z_{n-k}=\bar{z}_{k}$, and $z_{n}=\bar{z}_{n}$.

Using the isomorphism $Tz_{k}=(z_{k},\bar{z}_{k})$ for $k\notin\{n/2,n\}$, the
inclusion of $\mathbb{R}$ in $\mathbb{C}$ for $k=n/2$, and the inclusion of
$\mathbb{R}^{2}$ in $\mathbb{C}^{2}$ for $k=n$, one may prove that $M_{1}%
(\nu)$ in the fixed-point space of $\tilde{\kappa}$, is equivalent to the matrix%

\[
(m_{11},m_{12},...,m_{1(n/2)},m_{1n})\text{,}%
\]
defined in the space $\mathbb{R}^{2}$ for $m_{1(n/2)}$ and $m_{1n}$ and
$\mathbb{C}^{2}$ for the other blocks.

Applying a degree argument in the fixed-point space of $\tilde{\kappa}$, one
may prove bifurcation of periodic solutions for the blocks given by $m_{1n}$
and $m_{1n/2}$. The remaining blocks $m_{0k}$ for $k\notin\{n/2,n\}$ are
defined on complex subspaces, then they have always non-negative determinants,
as real matrices. Thus, there is no change of a standard degree in the
isotropy subspace.

Analytical studies with normal forms of high order and additional hypotheses
of non-resonance are proposed in \cite{CF08} for these cases. One could also
use the gradient structure and apply the results for bifurcation given in
\cite{Iz95}, p.100, based on Conley index. However, these results do not
provide the proof of the existence of a global continuum, something which
follows from the application of the orthogonal degree.

This fact implies that one may not use a classical degree argument or other
simple analytical proofs, and if one wishes to restrict the problem to one of
the isotropy subspaces described in this paper, one has to use anyway the
orthogonal degree. Our approach enables us to treat the complete problem in
one single system.

\subsection{The problem of $n$-charges}

We wish to analyze the movement of $n$ particles with charge $-1$ interacting
with a fixed nucleus with charge $q>0$. We suppose that the gravitational
forces are much smaller than Coulomb's forces. Thus, these equations may be a
classical model for the atom. Since electrons and protons have equal charge
with different signs, then $q=n$ for a non-ionized atom.

Let $x_{j}$ be the positions of the negative charges for $j\in\{1,...,n\}$,
then the equations describing the movement of the charges are%
\[
\ddot{x}_{j}=-q\frac{x_{j}}{\left\Vert x_{j}\right\Vert ^{3}}+\sum_{i=1~(i\neq
j)}^{n}\frac{x_{j}-x_{i}}{\left\Vert x_{j}-x_{i}\right\Vert ^{3}}\text{,}%
\]
where the first term represents the interaction with the fixed nucleus.

In rotating coordinates, $x_{j}(t)=e^{\sqrt{\omega}t\bar{J}}u_{j}(t)$, the
equations become%
\begin{align*}
\ddot{u}_{j}+2\sqrt{\omega}\bar{J}\dot{u}_{j}  &  =\nabla\tilde{V}(u)\text{
with}\\
\tilde{V}(u)  &  =\frac{\omega}{2}\sum_{j=1}^{n}\left\Vert \bar{I}%
u_{j}\right\Vert ^{2}-\sum_{i<j}\frac{1}{\left\Vert u_{j}-u_{i}\right\Vert
}+\sum_{j=1}^{n}\frac{q}{\left\Vert u_{j}\right\Vert }\text{.}%
\end{align*}
Let $u$ be $(u_{1},...,u_{n})$, then%
\[
\ddot{u}+2\sqrt{\omega}\mathcal{\bar{J}}\dot{u}=\nabla\tilde{V}(u)\text{.}%
\]

We wish to show the similarities of this problem with the $(n+1)$-body
problem. If we set $u_{0}=0$ in the potential of the $(n+1)$ bodies, the
potential for the bodies is%
\[
V(0,u)=\frac{\mu+s_{1}}{2}\sum_{j=1}^{n}\left\Vert \bar{I}u_{j}\right\Vert
^{2}+\sum_{j=1}^{n}\frac{\mu}{\left\Vert u_{j}\right\Vert }+\sum_{0<i<j}%
\frac{1}{\left\Vert u_{j}-u_{i}\right\Vert }\text{.}%
\]
We have proved before that the polygonal equilibrium $(0,a_{1},...,a_{n})$,
with $a_{j}=e^{ij\zeta}$, is a critical point of the potential $V(0,u)$.

If we choose $\omega=q-s_{1}$, in $\tilde{V}(u)$, and put $\mu=-q$ in $V(0,u)
$, then the potential for the $n$-charges satisfies%
\[
\tilde{V}(u)=-V(0,u).
\]
From this equality, we have that $\bar{a}=(a_{1},...,a_{n})$ is a critical
point of $\tilde{V}(u)$, with $\omega=q-s_{1}>0$, because $\bar{a}$ is a
critical point of $V(0,u)$.

Thus, we may perform a similar analysis to the bifurcation of periodic
solutions from the polygonal equilibrium in the body problem. The only
difference is in the spectrum. Hence, we shall focus on the spectrum of the
polygonal equilibrium for the charges.

The block corresponding to the Fourier modes of the charges is%
\[
\tilde{M}(\nu)=\nu^{2}I-2\nu\sqrt{\omega}(i\mathcal{\bar{J})}+D^{2}\tilde
{V}(\bar{a})\text{.}%
\]
Using the computation of $D^{2}V$ and the equality $D^{2}\tilde{V}%
(u;q)=-D^{2}V(0,u;-q)$, and after the change of variables, we obtain the
decomposition of $\tilde{M}(\nu)$ into two blocks. Next, the planar spectrum
is decomposed into the blocks%
\[
\tilde{m}_{0k}(\nu)=\nu^{2}I-2\nu\sqrt{\omega}(iJ)-B_{k}(-q)\text{,}%
\]
for $k\in\{1,...,n\}$, where $B_{k}(\mu)$ are the matrices given in the body
problem%
\[
B_{k}(\mu)=(3/2)(I+R)\mu+(s_{1}+\alpha_{k})I-\beta_{k}R-\gamma_{k}iJ\text{.}%
\]
While the spatial spectrum is decomposed into the blocks%
\[
\tilde{m}_{1k}(\nu)=\nu^{2}+(-q+s_{k})
\]
for $k\in\{1,...,n\}$.

For $k\in\{1,...,n\}$, if $q>s_{k}$, then the block $\tilde{m}_{1k}(\nu)$
changes its Morse index at the value
\[
\nu_{1k}=\sqrt{q-s_{k}}%
\]
with $\eta_{1k}(\nu_{1k})=1$.

\begin{proposition}
For each $k\in\{1,...,n\}$, the block $\tilde{m}_{0k}(\nu)$ changes its Morse
index at a positive value $\nu_{k}(\mu)$ for $q\in(s_{1},\infty)$ with%
\[
\eta_{0k}(\nu_{k})=1.
\]

\end{proposition}

\begin{proof}
For $k=n$, we have that $\det\tilde{m}_{0n}(\nu)=\omega^{2}\nu^{2}(\nu^{2}-1)
$, thus $\det\tilde{m}_{0n}(\nu)$ is zero only at $\nu_{n}=1$. For
$k\in\{1,...,n-1\}$, the determinant of $\tilde{m}_{0k}(\nu)$ is
\[
d_{k}(q,\nu)=\omega^{2}\nu^{4}-(2\alpha_{k}+\omega-s_{1})\omega\nu^{2}%
+4\omega\gamma_{k}\nu+a_{k}-qb_{k}%
\]
with $\omega=q-s_{1}$. Since $d_{k}(q,0)=a_{k}-qb_{k}$ is negative and
$d_{k}(\nu)$ is positive for large $\nu$, then there is a positive value
$\nu_{k}$ where $\tilde{m}_{0k}(\mu)$ changes its Morse index. For $q>s_{1}$,
we have $\sigma=1$ since $\tilde{m}_{0n}(0)=diag(3(q-s_{1}),0)$.

Since $d_{k}(q,0)=a_{k}-qb_{k}$ is negative, then $n_{k}(0)=1$, and since
$n(\infty)=0$, then $\eta_{k}(\nu_{k})=1-0$.
\end{proof}

\begin{theorem}
For $n\geq2$ and $q\in(s_{1},\infty)$, the polygonal equilibrium has a global
bifurcation of planar periodic solutions with isotropy group $\mathbb{Z}%
_{n}\left(  \zeta,\zeta,-k\zeta\right)  \times\mathbb{Z}_{2}(\kappa)$ for each
$k\in\{1,...,n\}$. If $q>s_{k}$, there is a global bifurcation of spatial
periodic solutions with isotropy group $\mathbb{Z}_{n}\left(  \zeta
,\zeta,-k\zeta\right)  \times\mathbb{Z}_{2}(\kappa,\pi)$ for each
$k\in\{1,...,n\}$ and starting at $\nu_{1k}$.
\end{theorem}

\begin{remark}
It is possible prove that the function $\nu_{k}(\mu)$ decreases from $\nu
_{k}(s_{1})=\infty$ to $\nu_{k}(\infty)=1$, and that $\nu_{k}$ is unique for
$k\in\{2,...,n-2,n\}$, and for $k\in\{1,n-1\}$ if $n\geq7$. In those cases the
bifurcation is non-admissible or goes to another equilibrium. Note also that,
if $q=n$, then one may verify numerically that $s_{1}<q$ only for $n<473$.
\end{remark}

\section{Appendix}

Let $\zeta=2\pi/n$, for a fixed $\alpha\in\lbrack1,\infty)$, we define the
sums $s_{k}$ and $\bar{s}_{k}$ as%
\[
s_{k}=\frac{1}{2^{\alpha}}\sum_{j=1}^{n-1}\frac{\sin^{2}(kj\zeta/2)}%
{\sin^{\alpha+1}(j\zeta/2)}\text{ and }\bar{s}_{k}=\frac{1}{2^{\alpha-2}}%
\sum_{j=1}^{n-1}\frac{\sin^{2}(kj\zeta/2)}{\sin^{\alpha-1}(j\zeta/2)}\text{.}%
\]
Then, the sums $s_{k}$ are always positive and satisfy
\[
s_{k}=s_{n+k}=s_{n-k}.
\]

\begin{proposition}
\label{EA.1.0}The sums $s_{k}$\ satisfy the following three recurrence
formulae
\[
s_{k+1}-2s_{k}+s_{k-1}=2s_{1}-\bar{s}_{k}%
\]
\[
s_{k}=k^{2}s_{1}-\sum_{l=1}^{k-1}l\bar{s}_{k-l}\text{ and }%
\]
\[
s_{k+1} - s_{k}=(2k+1)s_{1}-\sum_{l=1}^{k}\bar{s}_{l}\text{.}%
\]

\end{proposition}

\begin{proof}
Write $s_{k}$ as%
\[
2^{\alpha}s_{k}=\sum_{j=1}^{n-1}\frac{1}{\sin^{\alpha-1}(j\zeta/2)}%
\frac{1-\cos(kj\zeta)}{1-\cos(j\zeta)}\text{.}%
\]
Using geometric series, we obtain that%
\[
\frac{1-\cos(kj\zeta)}{1-\cos(j\zeta)}=\frac{1-e^{ijk\zeta}}{1-e^{ij\zeta}%
}\frac{1-e^{-ijk\zeta}}{1-e^{-ij\zeta}}=\sum_{l=0}^{k-1}\sum_{m=0}%
^{k-1}e^{ij(l-m)\zeta}.
\]
Therefore, the difference of the sums satisfies%
\begin{equation}
2^{\alpha}(s_{k+1}-s_{k})=\sum_{j=1}^{n-1}\frac{1}{\sin^{\alpha-1}(j\zeta
/2)}\sum_{h=-k}^{k}e^{ijh\zeta}\text{.} \label{DA.1.1}%
\end{equation}

Since the sum of exponents is%
\[
\sum_{h=-k}^{k}e^{ijh\zeta}=\sum_{h=-k}^{k}\cos jh\zeta=\sum_{h=-k}%
^{k}(1-2\sin^{2}(jh\zeta/2))=(2k+1)-4\sum_{h=1}^{k}\sin^{2}(jh\zeta/2)\text{,}%
\]
then,%
\[
s_{k+1}-s_{k}=(2k+1)s_{1}-\sum_{h=1}^{k}\sum_{j=1}^{n-1}\frac{\sin^{2}%
(hj\zeta/2)}{2^{\alpha-2}\sin^{\alpha-1}(j\zeta/2)}=(2k+1)s_{1}-\sum_{h=1}%
^{k}\bar{s}_{h}%
\]
Iterating the previous formula we obtain the other two recurrence formulae.
\end{proof}

We have used the results about the $s_{k}$'s for the vortex problem,
$\alpha=1$, in our paper \cite{GaIz12}. Now, we may conclude the following corollary.

\begin{corollary}
Using the fact that $\bar{s}_{k}$ and $4s_{k}-\bar{s}_{k}$ are positive, one
has%
\[
s_{k+1}-2s_{k}+s_{k-1}<2s_{1}\text{ and }s_{k+1}+2s_{k}+s_{k-1}>2s_{1}\text{.}%
\]

\end{corollary}

For the $n$-body problem, $\alpha=2$, we cannot find explicitly the sums
$s_{k}$. However, we may prove that\ the sums $s_{k}$ are increasing, for
$k\in\{0,...,[n/2]\}$.

\begin{lemma}
Let $s$ be the sum $s=-\sum_{j=0}^{n-1}\sin(k-1/2)j\zeta$, then $s=-\cot
(k-1/2)\zeta/2$.
\end{lemma}

\begin{proof}
Let $\xi=(k-1/2)\zeta$, then
\[
s=-\sum_{j=0}^{n-1}\sin j\xi=-\frac{1}{2i}\sum_{j=0}^{n-1}\left(  e^{i\xi
j}-e^{-i\xi j}\right)  .
\]
Using geometric series, we have%
\begin{align*}
s  &  =-\frac{1}{2i}\left(  \frac{\left(  1-e^{i\xi n}\right)  }{\left(
1-e^{i\xi}\right)  }-\frac{\left(  1-e^{-i\xi n}\right)  }{\left(  1-e^{-i\xi
}\right)  }\right) \\
&  =-\frac{1}{2i}\left(  \frac{e^{i\xi}-e^{-i\xi}-(e^{in\xi}-e^{-in\xi
})+e^{i(n-1)\xi}-e^{-i(n-1)\xi}}{2-e^{i\xi}-e^{-i\xi}}\right)  \text{.}%
\end{align*}

Thus,%
\[
s=-\frac{\sin\xi-\sin n\xi+\sin(n-1)\xi}{2(1-\cos\xi)}\text{.}%
\]
Since $n\xi=2\pi(k-1/2)\equiv\pi$ modulus $2\pi$, then $\sin n\xi=0$ and
$\sin(n-1)\xi=\sin\xi$. Hence,%
\[
s=-\frac{\sin\xi}{1-\cos\xi}=-\cot(\xi/2)\text{.}%
\]

\end{proof}

\begin{proposition}
The sums $s_{k}$ for $k\in\{1,...,[n/2]\}$ are increasing,%
\[
s_{k}>s_{k-1}.
\]

\end{proposition}

\begin{proof}
Define $d_{k}$, $d_{k}^{\prime}$ and $d_{k}^{\prime\prime}$ as%
\[
d_{k}=s_{k+1}-s_{k}\text{, }d_{k}^{\prime}=d_{k}-d_{k-1}\text{ and }%
d_{k}^{\prime\prime}=d_{k}^{\prime}-d_{k-1}^{\prime}\text{.}%
\]
The goal is to prove that $d_{k}^{\prime\prime}$ are negative, and that
$d_{k}$ are positive, for $k\in\{1,...,[n/2]\}$.

From the equality (\ref{DA.1.1}), we have, for $\alpha=2$, that%
\[
d_{k}=\frac{1}{4}\sum_{j=1}^{n-1}\frac{1}{\sin(j\zeta/2)}\sum_{h=-k}%
^{k}e^{ijh\zeta}\text{.}%
\]
Therefore, the differences $d_{k}^{\prime}$ are%
\[
d_{k}^{\prime}=\frac{1}{4}\sum_{j=1}^{n-1}\frac{2\cos jk\zeta}{\sin(j\zeta
/2)}\text{.}%
\]
Thus,%
\[
d_{k}^{\prime\prime}=\frac{1}{2}\sum_{j=1}^{n-1}\frac{\cos jk\zeta-\cos
j(k-1)\zeta}{\sin(j\zeta/2)}\text{.}%
\]

Now, we wish to compute the sum $d_{k}^{\prime\prime}$ and prove that it is
negative. For this, we need the following trigonometric identity%
\begin{align*}
\cos jk\zeta-\cos j(k-1)\zeta &  =(1-\cos j\zeta)\cos jk\zeta-\sin jk\zeta\sin
k\zeta\\
&  =2\sin\frac{j\zeta}{2}\left(  \sin\frac{j\zeta}{2}\cos jk\zeta-\sin
jk\zeta\cos\frac{k\zeta}{2}\right) \\
&  =-2\sin\frac{j\zeta}{2}\sin(k-\frac{1}{2})j\zeta\text{.}%
\end{align*}
Therefore,
\[
d_{k}^{\prime\prime}=-\sum_{j=0}^{n-1}\sin(k-1/2)j\zeta\text{.}%
\]

From the previous lemma, we have that $d_{k}^{\prime\prime}=-\cot(\xi/2)$,
then $d_{k}^{\prime\prime}$ is negative when $\xi=\pi(2k-1)/n\in(0,\pi)$.
Therefore, the numbers $d_{k}^{\prime\prime}$ are negative for $k\in
\mathbb{N}\cap(1/2,(n+1)/2)$, and so the numbers $d_{k}^{\prime}$ decrease for
$k\in\mathbb{N}\cap\lbrack0,n/2]$.

Since $d_{0}^{\prime}=2s_{1}$ is positive and $d_{k}^{\prime}$ decreases for
$k\in\mathbb{N}\cap\lbrack0,n/2]$, there can be at most one $k_{0}%
\in\mathbb{N}\cap\lbrack1,n/2]$ such that%
\[
d_{k_{0}-1}^{\prime}>0\geq d_{k_{0}}^{\prime}.
\]
Suppose first that $k_{0}$ does not exist. Then $d_{k}^{\prime}$ is positive
for all $k\in\mathbb{N}\cap\lbrack0,n/2]$, and $d_{k}$ increases for all
$k\in\mathbb{N}\cap\lbrack0,n/2]$. Therefore, $d_{k}>d_{0}>0$ for all
$k\in\mathbb{N}\cap\lbrack0,n/2]$, because $d_{0}=s_{1}$.

Now, suppose that $k_{0}$ does exist. Using that $s_{n-k}=s_{k}$, we may prove
that $d_{k}$ satisfies the equality%
\begin{equation}
d_{k-1}=(s_{k}-s_{k-1})=-(s_{(n-k)+1}-s_{(n-k)})=-d_{n-k}\text{.}
\label{DA.1.2}%
\end{equation}
If $n$ is odd, from the equality (\ref{DA.1.2}) with $k=(n+1)/2$, we have that
$d_{(n-1)/2}=-d_{(n-1)/2}=0$. If $n$ is even, from the equality (\ref{DA.1.2})
with $k=n/2$ we have that $d_{n/2-1}=-d_{n/2}$. Since we did suppose that the
$k_{0}$ exists, then%
\[
2d_{n/2-1}=d_{n/2-1}-d_{n/2}=-d_{n/2}^{\prime}>0.
\]
Finally, since $d_{k}$ increases in $\mathbb{N}\cap(0,k_{0})$ and decreases in
$\mathbb{N}\cap(k_{0},[(n-1)/2]]$, then the numbers $d_{k}$ are positive for
$k\in\{0,...,[(n-1)/2]\}$, because $d_{k}$ is positive at the end points
$k=0,[(n-1)/2]$. In any case, the sums $s_{k}$ increase for $k\in
\{0,...,[n/2]\}$.
\end{proof}

\begin{acknowledgement}
The authors wish to thank the referees for their comments, for providing a
list of additional references on the variational approach to the $n$-body
problem and for asking for more precise references. Also, C.G-A wishes to
thank the CONACyT for his scholarship and J.I for the grant No. 133036.
\end{acknowledgement}

\end{document}